\documentclass{amsart}
\usepackage{graphicx}

\newtheorem{theorem}{Theorem}[section]
\newtheorem{lemma}[theorem]{Lemma}
\newtheorem{prop}[theorem]{Proposition}
\newtheorem{cor}[theorem]{Corollary}

\newtheorem{ques}[theorem]{Question}

\newtheorem{defn}[theorem]{Definition}

\newcommand{\Gab}{G  -  a,b}

\newcommand{\R}{\mathbb R}

\newcommand{\TY}{\nabla\mathrm{Y}}
\newcommand{\YT}{\mathrm{Y}\nabla}

\title{Graphs on 21 edges that are not $2$--apex}
\author{Jamison Barsotti}
\address{Department of Mathematics,
University of California Santa Cruz, Santa Cruz, CA 95064}
\email{jbarsott@ucsc.edu}
\author{Thomas W.\ Mattman}
\address{Department of Mathematics and Statistics,
California State University, Chico,
Chico, CA 95929-0525}
\email{TMattman@CSUChico.edu}

\begin{document}

\begin{abstract}
We show that the 20 graph Heawood family, obtained by a combination of $\TY$ and $\YT$ moves on $K_7$, is precisely the set of graphs of at most 21 edges that are minor minimal for the property not $2$--apex. As a corollary, this gives a new proof that the 14 graphs obtained by $\TY$ moves on $K_7$ are the minor minimal intrinsically knotted graphs of 21 or fewer edges. Similarly, we argue that the seven graph Petersen family, obtained from $K_6$, is the set of graphs of at most 17 edges that are minor minimal for the property not apex.
\end{abstract}

\maketitle

\section{Introduction}

A graph is {\bf $n$--apex} if the deletion of $n$ or fewer vertices results in a planar graph. As this property is
closed under taking minors, it follows from Robertson and Seymour's Graph Minor Theorem~\cite{RS} 
that, for each $n$, the $n$--apex graphs are characterized by a finite set of forbidden minors. For example, $0$--apex is equivalent to planarity, which Wagner~\cite{W} showed is characterized by $K_5$ and $K_{3,3}$. For the property $1$--apex,  which we simply call {\bf apex}, there are several hundreds of forbidden graphs (see \cite{DD}, which refers to work of a team led by Kezdy). Since there are likely even more forbidden minors for the $2$--apex property, we divide the problem into more manageable pieces by graph size. In an earlier paper~\cite{Ma}, the second author showed that every graph on 20 or fewer edges is $2$--apex. This means there are no forbidden minors with 20 or fewer edges. In the current paper, we show that there are exactly 20 obstruction graphs for $2$--apex of size at most 21.

Following~\cite{HNTY}, the  {\bf Heawood family} will denote the set of 20 graphs obtained from $K_7$ by a sequence of zero or more $\TY$ or $\YT$ moves.  
Recall that a {\bf $\TY$ move} consists of deleting the edges of a $3$-cycle $abc$ of graph $G$, and adding a new degree three vertex adjacent to the vertices $a$, $b$, and $c$. The reverse, deleting a degree three vertex and making its neighbors adjacent, is a {\bf $\YT$ move}.
The Heawood family is illustrated schematically in Figure~\ref{figHea}  (taken from \cite{GMN}) where $K_7$ is graph 1 at the top of the figure and the $(14,21)$ Heawood graph is graph 18 at the bottom. 

\begin{figure}[htb]
\begin{center}
\includegraphics[scale=0.75]{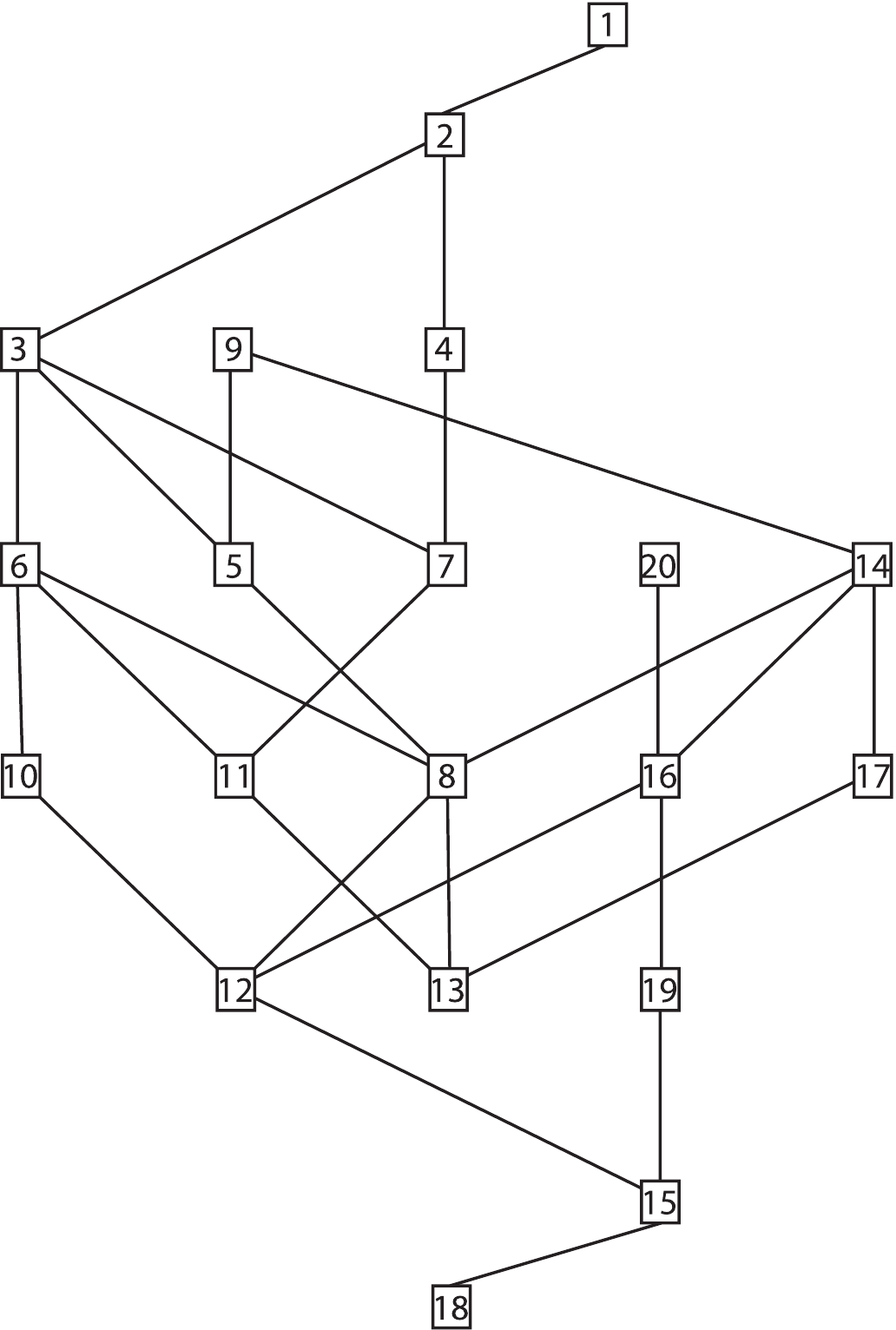}
\caption{The Heawood family (figure taken from \cite{GMN}). Edges represent $\TY$ moves.}\label{figHea}
\end{center}
\end{figure}

Our main theorem is that the Heawood family is precisely the obstruction set for the property $2$--apex among graphs of size at most 21. We will state this in terms of minor minimality. We say $H$ is a {\bf minor} of graph $G$ if $H$ is obtained by contracting edges in a subgraph of $G$. The graph $G$ is {\bf minor minimal} with respect to a graph property $\mathcal P$, if $G$ has $\mathcal P$, but no proper minor of $G$ does. 
We call obstruction graphs for the $2$--apex property {\bf minor minimal not $2$--apex or MMN2A}.

\begin{theorem}  \label{thmain}%
The 20 Heawood family graphs are the only MMN2A graphs on 21 or fewer edges.
\end{theorem}

As there are no MMN2A graphs of size 20 or less \cite{Ma} and one easily verifies that the Heawood family graphs are MMN2A, the argument comes down to showing no other 21 edge graph enjoys this property. We give a more complete outline of our proof at the end of this introduction.

Our interest in  $2$--apex stems from the close connection with intrinsic knotting.
A graph is {\bf intrinsically knotted or IK} if every tame embedding of the graph in $\R^3$ contains a non-trivially knotted cycle. Then, a  {\bf minor minimal IK or MMIK} graph is one that is IK, but such that no proper minor has this property. Again, Robertson and Seymour's Graph Minor Theorem~\cite{RS} implies 
a finite list of MMIK graphs, but determining this list or even bounding its size has proved very difficult. 
Restricting by order, it follows from Conway and Gordon's seminal paper~\cite{CG} that $K_7$ is the only MMIK graph on seven or fewer vertices; two groups~\cite{CMOPRW} and \cite{BBFFHL} independently determined the MMIK graphs of order eight; and we have announced (see \cite{Mo} and \cite{GMN}) a classification of nine vertex graphs, based on a computer search. In terms of edges, we know (\cite{JKM} and, independently, \cite{Ma}) that a graph of size 20 or less is not IK. Using the following lemma, (due, independently, to two research teams) this follows from the lack of MMN2A graphs of that size.

\begin{lemma} \label{lem2ap}%
\cite{BBFFHL, OT}
If $G$ is IK, then $G$ is not $2$--apex.
\end{lemma}

The current authors~\cite{BM} and, independently, Lee et al.~\cite{LKLO} classified the 21 edge MMIK graphs. These are the 
14 KS graphs obtained by $\TY$ moves on $K_7$, first described by Kohara and Suzuki~\cite{KS}. In other words, these are the Heawood family graphs except those labelled 9, 14, 16, 17, 19, and 20 in Figure~\ref{figHea}. In light of Lemma~\ref{lem2ap}, we have a new proof as a corollary to our main theorem.

\begin{cor} \label{corMMIK} 
The 14 KS graphs are the only MMIK graphs on 21 or fewer edges.
\end{cor}

\begin{proof}
Kohara and Suzuki~\cite{KS}  showed that the KS graphs are MMIK. 
Suppose $G$ is MMIK of at most 21 edges. Then $G$ is connected. By Lemma~\ref{lem2ap}, $G$ has an MMN2A minor and by Theorem~\ref{thmain} this means a Heawood family graph minor. As $G$ has at most 21 edges and is connected,  $G$ {\em is} a Heawood family graph. Finally, Goldberg et al.~\cite{GMN} and Hanaki, Nikkuni, Taniyama, and Yamazaki~\cite{HNTY}, independently, showed that in the Heawood family only the KS graphs are IK. Therefore, $G$ is a KS graph.
\end{proof}

The proof of our main theorem relies on our classification of MMNA graphs (i.e., obstructions to the $1$--apex, or apex, property) of small order, a result that may be of independent interest.
Recall that, in analogy with the Heawood family, the Petersen family is the seven graphs obtained from the Petersen graph by a sequence of $\TY$ or $\YT$ moves.

\begin{theorem} \label{thmMMNA}
The seven Petersen family graphs are the only MMNA graphs on 16 or fewer edges
\end{theorem}

Famously, the Petersen family is precisely the obstruction set to intrinsic linking~\cite{RST}. It would be nice to have a similar description of
the Heawood family. Theorem~\ref{thmain} is one such characterization. As a second corollary to our main theorem, we give a characterization of similar flavor. Hanaki, Nikkuni, Taniyama, and Yamazaki~\cite{HNTY} showed that the Heawood family graphs are 
minor minimal for intrinsically knotted or completely 3-linked or MMI(K or C3L). 

\begin{cor} \label{corHNTY}
The 20 Heawood family graphs are the only MMI(K or C3L) graphs on 21 or fewer edges.
\end{cor}

\begin{proof}
Hanaki et al.\ \cite{HNTY} proved these graphs are MMI(K or C3L).
Let $G$ be MMI(K or C3L) on 21 or fewer edges. Then $G$ is connected.
By \cite[Remark 4.5] {HNTY},  I(K or C3L) implies N2A, so $G$ must have a MMN2A minor. By Theorem~\ref{thmain}, this means a Heawood minor.
It follows that $G$ has 21 edges and is a Heawood family graph, as required.
\end{proof}

This gives two characterizations of the Heawood family. However, like our Theorem~\ref{thmMMNA}, they are less than ideal due to the hypothesis on graph size. Is there a ``natural'' description of the Heawood family analogous to the way the 
Petersen family is precisely the obstruction set for intrinsic linking?

Note that the condition on graph size in these three results is necessary. Indeed, for 
Theorem~\ref{thmMMNA}, the disjoint union $K_{3,3} \sqcup K_{3,3}$ is an 18 edge MMNA graph outside the Petersen family. On the other hand, a computer search~\cite{P}
shows that Theorem~\ref{thmMMNA} could be extended to 17 edges: there are no
MMNA graphs of size 17. Since IK implies 
both N2A (Lemma~\ref{lem2ap}) and I(K or C3L) (see ~\cite{HNTY}) there are many examples of MMN2A and MMI(K or C3L) graphs on 22 edges, including $K_{3,3,1,1}$.
 Foisy~\cite{F} showed this graph is MMIK, which means it is also N2A and I(K or C3L). 
As any proper minor of $K_{3,3,1,1}$ would have at most 21 edges, and no Heawood family graph is a minor, it follows from Theorem~\ref{thmain} and Corollary~\ref{corHNTY}, that $K_{3,3,1,1}$ is both MMN2A and MMI(K or C3L).
So, the hypothesis on size is necessary for both the theorem and its corollary.
 
Thus, $K_{3,3,1,1}$ and the 14 KS graphs are examples of graphs that enjoy all three properties: MMN2A, MMIK, and MMI(K or C3L). On the other hand, the remaining six 
Heawood graphs show that a graph can be MMN2A and not MMIK. This includes the graph that we have called $E_9$  \cite{Ma} and that Hanaki et al.~\cite{HNTY} label $N_9$. In \cite{GMN} we showed that adding an edge to this graph makes it MMIK. In other words, $E_9+e$ is MMIK and not MMN2A (as it has the N2A graph $E_9$ as a subgraph). On the other hand, since IK implies I(K or C3L), every MMIK graph has a minor that is MMI(K or C3L) although $E_9$, for example, shows that 
the set of I(K or C3L) graphs is a strictly larger class than IK.
Similarly,  
 I(K or C3L) implies N2A \cite{HNTY}, which means every MMI(K or C3L) has a MMN2A minor, while the  disjoint union of three $K_{3,3}$'s is an example of a graph
that is N2A but not I(K or C3L).

All six of the Heawood graphs that are not MMIK are MMI(K or C3L) and we can ask if a graph that is MMN2A and not MMIK need be I(K or C3L).  However, 
the disjoint union $G = K_6 \sqcup K_5$ is a counterexample. Since $K_6$ is MMNA and $K_5$ is non-planar, $G$ is N2A and, since any proper minor is $2$--apex, 
it is in fact MMN2A. On the other hand, $G$ is neither IK nor I(K or C3L) as each component has fewer than 21 edges.

We conclude this overview of connections between apex graphs and intrinsic knotting with a question. In ~\cite{GMN} we describe the known 263 examples of MMIK graphs. By Lemma~\ref{lem2ap}, none of these graphs are $2$--apex. However, it is straightforward to verify that each is $3$--apex. Does this hold more generally?

\begin{ques} Is every MMIK graph $3$--apex? \end{ques}

The remainder of our paper is a proof of Theorem~\ref{thmain}. Let $G$ be a MMN2A graph of size 21. We must show $G$ is a Heawood family graph.
We can assume $\delta(G)$, the {\bf minimum degree}, is at least three. Indeed, in a N2A graph, deleting a degree zero vertex or
contracting an edge of a vertex of degree one or two will result in a N2A minor. 
We can also bound the number of vertices.
As $G$ has 21 edges and minimum degree at least three it has at most 14 vertices. 
On the other hand, we classified MMN2A graphs on nine or
fewer vertices in \cite{Ma}. So we can assume $10 \leq |V(G)| \leq 14$.
After introducing some preliminary lemmas, and proving Theorem~\ref{thmMMNA},  in the next section, we devote one section each to the five cases where the number of vertices runs from 14 down to ten. We opted for this reverse ordering as 
it roughly corresponds to increasing length of the proofs.

\section{Preliminaries}

We denote the order of a graph $G$ by $|G|$ and its size by $\|G\|$ and frequently use the pair $(|G|, \|G\|)$ as a way of describing the graph.
For $a,b \in V(G)$, we will use $G-a$ and $G-a,b$ to denote the induced subgraphs on $V(G) \setminus \{a\}$ and $V(G) \setminus \{a,b \}$, respectively.
We will also write $G+a$ to denote a graph with vertices $V(G) \cup \{a\}$ that includes $G$ as the induced subgraph on $V(G)$. In case $V(G)$ and $\{a\}$ are included in the vertex set of some larger graph, $G+a$ will mean the induced subgraph on $V(G) \cup \{a\}$.
We use $N(a)$ to denote the \textbf{neighborhood of vertex} $a$, 
the set of vertices adjacent to $a$.
We will write NA, MMNA, N2A, and MMN2A for ``not apex",
(equivalently, ``not $1$--apex")  ``minor minimal not apex", ``not $2$--apex", and ``minor minimal not $2$--apex" respectively. 

Vertices of degree less than three do not participate in determining whether or not a graph is $n$--apex, so we next describe a systematic way of deleting those vertices. 

\begin{defn} The {\bf simplification} $G^s$ of a graph $G$ is the graph obtained by the following procedure.
\begin{enumerate}
\item Delete all degree 0 vertices
\item Delete all degree 1 vertices and their edges
\item If there remain vertices of degree 0 or 1, go to step (1)
\item For each degree 2 vertex $v$, delete it and its two edges $va$ and $vb$. If $ab$ is not already an edge of the graph, add $ab$.
\item If there remain any vertices of degree 0 or 1, go to step (1)
\end{enumerate}
The procedure allows us to recognize $V(G^s)$ as a subset of $V(G)$. We call these vertices  of $G$ the {\bf  branch vertices}.
\end{defn}

Note that $G^s$ is a minor of $G$ and is unique, up to isomorphism \cite{P}. 

\begin{lemma} The graph $G$ is $n$--apex if and only if $G^s$ is.
\end{lemma}

\begin{proof} This follows as $n$--apex is preserved by each step in the definition.
\end{proof}

This means that graphs where $G^s$ is non-planar will be of particular interest. An important class of graphs with $G^s = K_{3,3}$ are the
{\bf split $K_{3,3}$'s}: graphs obtained from
$K_{3,3}$ by a finite (possibly empty) sequence of vertex splits.

In this section, we will prove Theorem~\ref{thmMMNA}, the Petersen family graphs are the MMNA graphs with $\|G \| \leq 16$. Recall that the Petersen family is the set of seven graphs obtained by $\TY$ and $\YT$ moves on the $(10,15)$ Petersen graph $P_{10}$. In addition to $P_{10}$, the set includes $K_6$, $K_{3,3,1}$, $K_{4,4}-e$, and, by definition, is closed under $\TY$ and $\YT$ moves. 
We first observe that each graph in the family is MMNA.

\begin{lemma} \label{lemPMMNA}
The seven graphs in the Petersen family are all MMNA
\end{lemma} 

\begin{proof}
Aside from describing what is to be checked, we omit most of the details.
Let $G$ be a graph in the Petersen family. It's enough to verify that $\forall v \in V(G)$,
$G-v$ is non-planar and that $\forall e \in E(G)$, deletion and contraction of $e$ both result in apex graphs.
\end{proof}

The proof of Theorem~\ref{thmMMNA} depends on the following lemma that  characterises NA graphs using the idea of a vertex near  a branch vertex.
If $G$ is a graph
and $w \in V(G)$ is such that there is a path from $w$ to  a branch vertex, $a$, of $G$ that contains 
no other  branch vertices of $G$, then we say $w$ is {\bf near} $a$.
Similarly, if $w$ is a vertex in some $G+v$, $w$ is near  a branch vertex $a$ of $G$ 
if there is a $w$-$a$ path independent of the other  branch vertices.

\begin{lemma} \label{lemapexpath}
Suppose $G$ simplifies to $K_5$ or $K_{3,3}$. Then $G+v$ is NA if and only if $v$ is near every  branch vertex of $G$.
\end{lemma}

\begin{proof}
As in the definition above, forming $G^s$, the simplification of $G$, determines a set of branch vertices.

First, assume that $G+v$ is NA and $v$ is not near  a branch vertex of $G$, call it $a$. If we remove  a branch vertex near $a$, call it $b$, then, we claim, $G+v-b$ is planar, which contradicts $G+v$ NA. To verify the claim, note that $G^s$ is minor minimal non-planar. The only way that $G+v-b$ could be non-planar would be for $v$ to take the place of $b$ in that graph. This would require independent paths from $v$ to each of the  branch vertices near $b$. As there is no such $v$-$a$ path,
$G+v-b$ is planar.

Now assume that, in $G+v$, $v$ is near every  branch vertex of $G$. Then $G^* = (G+v)^s$
is of the form $H + v$ where $H$ is a subdivision of $G^s$ and, by abuse of notation,
we again refer to the vertices of $H$ of degree three or four as  branch vertices (of $G$). In $G^*$, the neighbors of $v$ are either 
 branch vertices of $G$ or on edges of $G^s$ that were subdivided to form $H$.
In particular,  $v$ is near the same  branch vertices in $H+v$ as it was in $G+v$.
We wish to show that $G^*$ can, through a series of $\YT$ moves, be transformed into an NA graph. If, in $G^*$, $v$ is adjacent to all the  branch vertices of $G$, we are done, since if $G^s=K_5$, then $G^*$ has a $K_6$ minor, and if $G^s=K_{3,3}$ then $G^*$ has $K_{3,3,1}$
as a minor. As $K_6$ and $K_{3,3,1}$ are both NA (see previous lemma), $G+v$ is as well.

Next, choose  a branch vertex from $G$, call it $a$. Suppose $v$ is not adjacent to $a$ in $G^*$. However,  we've assumed $v$ is near every  branch vertex, including $a$. Hence there is a vertex of degree three that has
both $a$ and $v$ as neighbors, call it $w$. Performing a $\YT$ move on $w$ makes $a$ and $v$ neighbors and will not change the nearness of $v$ with
any  branch vertices. Repeating this process for the rest of the  branch vertices results in a graph where $v$ is adjacent to each  branch vertex of $G$.
Again, if $G^s = K_5$, then this series of $\YT$ moves on $G^*$ gives a graph that has a $K_6$ minor. If $G^s  = K_{3,3}$ then a 
series of $\YT$ moves on $G^*$ gives us a graph that has $K_{3,3,1}$ as a minor.
Since $\YT$ and $\TY$ preserve the Petersen family, we conclude that $G+v$ has a 
minor from the Petersen family and is, therefore, NA.
\end{proof}

The proof shows that, not only is $G+v$ NA, it has a Petersen family graph as a minor. On the other hand, if $G+v$ has a Petersen family graph minor, then it is NA by Lemma~\ref{lemPMMNA}. Also, Petersen family graph minors characterize intrinsic linking~\cite{RST}. The following lemma combines these observations.

\begin{lemma} \label{lemapathequiv}
Let $G$ be a graph with vertex $v$ such that $(G-v)^s = K_5$ or $K_{3,3}$. Then the following are equivalent.
\begin{itemize}
\item The vertex $v$ is near every branch vertex of $G-v$.
\item $G$ is NA.
\item $G$ has a Petersen family graph minor.
\item $G$ is intrinsically linked.
\end{itemize}
\end{lemma}

\begin{figure}[ht]
\begin{center}
\includegraphics[scale=0.5]{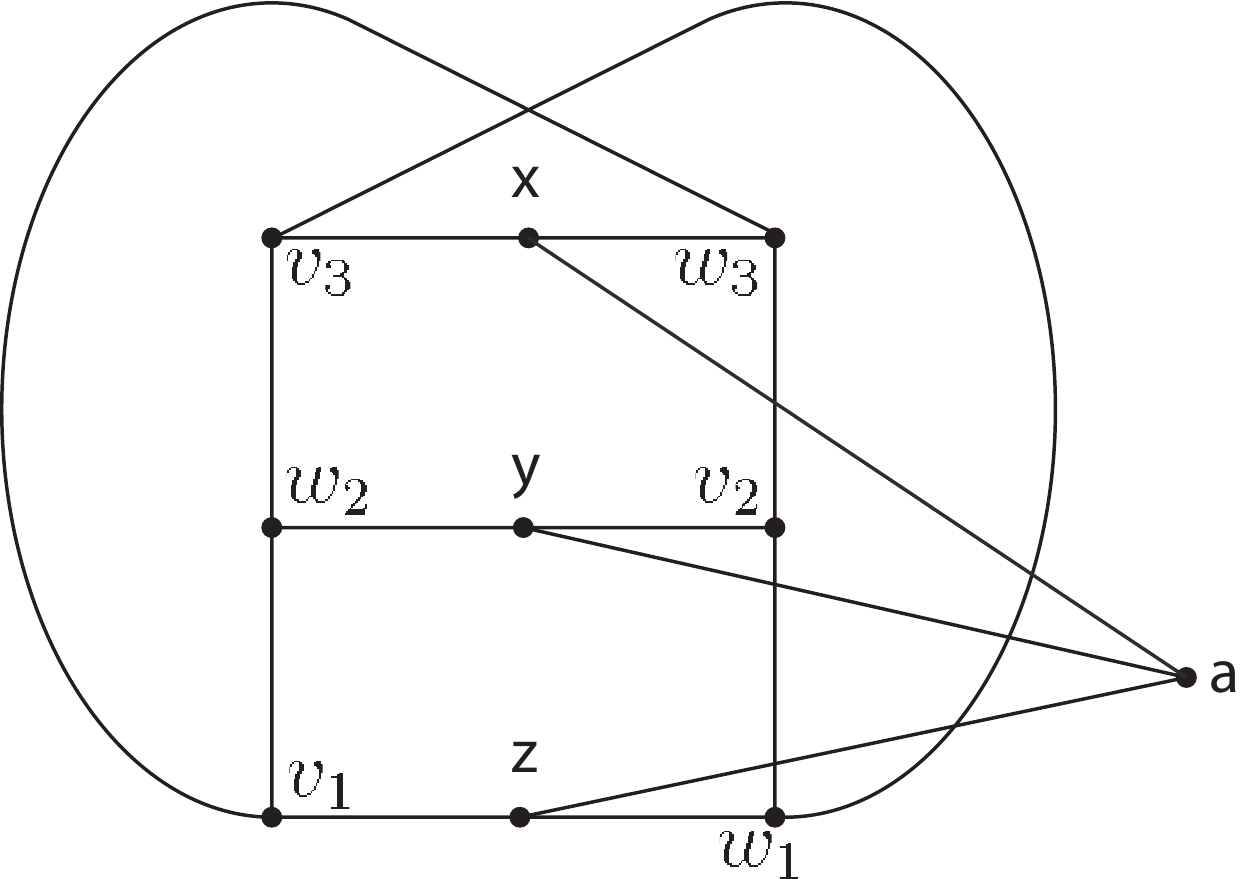}
\caption{Adding a degree $3$ vertex to a split $K_{3,3}$ yields the Petersen graph.}\label{figda3}
\end{center}
\end{figure}

\begin{lemma}
\label{lemda3}%
If $G+a$ is formed by adding a degree three vertex $a$ to a split $K_{3,3}$ graph $G$ and $G+a$ is NA, then $(G+a)^s$ is the Petersen graph.
\end{lemma}

\begin{proof}
By Lemma~\ref{lemapexpath}, there are paths from $a$ to each  branch vertex that avoid all other  branch vertices. Up to isomorphism, the only way to arrange this is as in the graph of Figure~\ref{figda3}, which is the Petersen graph.
\end{proof}

Figure~\ref{figda3} illustrates the idea of a vertex being near an edge. Let $G$ be such that $G^s = K_{3,3}$ or $K_5$. As in the proof of Lemma~\ref{lemapexpath}, 
if we add a vertex $v$, then, in general, $(G+v)^s$ will be of the form $H + v$ where $H$ is a subdivision of $G^s$. We say that $v$ is {\bf near the edge} $xy$ in $G^s$, where $x$ and $y$ are  branch vertices, if, in $(G+v)^s$, $v$ has a neighbor interior
to the (subdivided) edge $xy$ of $G^s$. In Figure~\ref{figda3}, $a$ is near the edges
$v_iw_i$, $i = 1,2,3$.

\begin{figure}[ht]
\begin{center}
\includegraphics[scale=.5]{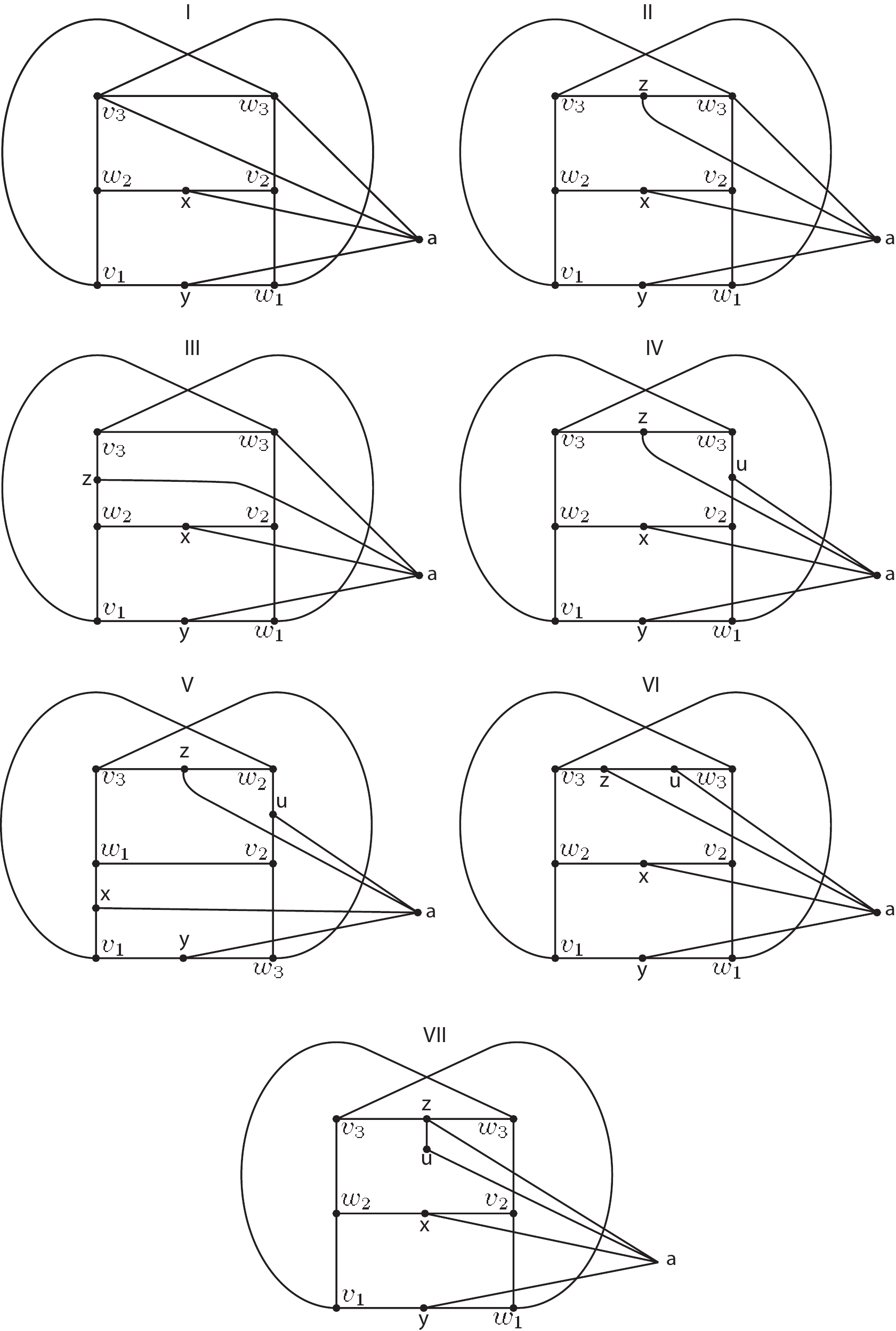}
\caption{Adding a degree $4$ vertex to a split $K_{3,3}$.}\label{figda4}
\end{center}
\end{figure}

\begin{lemma}
\label{lemda4}%
If $G+a$ is formed by adding a vertex $a$ of degree four to a split $K_{3,3}$ graph $G$ and $G+a$ is NA, then $(G+a)^s$ is one of the seven graphs in Figure~\ref{figda4}.
\end{lemma}

\begin{proof}
By Lemma~\ref{lemapexpath}, there are paths from $a$ to each branch vertex that avoid all other branch vertices. 
Let $N(a) = \{n_1,n_2,n_3,n_4\}$.
As there are six vertices and $d(a) = 4$, then there is an 
$n_i$, say $n_1$, that has an edge, say $v_1w_1$, as its nearest part. Since there are four branch vertices left and three neighbors of $a$, another $n_i$, say $n_2$, must have an edge
as its nearest part with vertices disjoint from $\{v_1, w_1\}$, call it $v_2 w_2$. 
There are three graphs generated
when $a$ has a neighbor whose nearest part is a branch vertex of $G$ and four more when $a$ has no such neighbor. Figure~\ref{figda4} shows the graphs that results from this condition.
\end{proof}

We conclude this section with a proof of Theorem~\ref{thmMMNA}. The proof requires one additional lemma. Let $\delta(G)$ and $\Delta(G)$ denote the 
\textbf{minimum} and \textbf{maximum degree} of graph $G$.

\begin{lemma} \label{lem34mix}
Suppose $G$ has $\delta(G) = 3$, $\Delta(G) = 4$, and $13 \leq \|G \| \leq 16$. 
Then either there is a degree $4$ vertex with a degree $3$ neighbor or else 
$G$ is the disjoint union $K_5 \sqcup K_4$.
\end{lemma}

\begin{proof} 
For a contradiction, suppose no degree $4$ vertex has a degree $3$ neighbor. 
Then $G$ is disconnected with cubic and quartic components. The smallest quartic graph is $K_5$ with ten edges and the smallest cubic graph is $K_4$ with six.
So, the order of $G$ is at least 16 and $K_5 \sqcup K_4$ is the only way to realize that minimum. 
\end{proof}

\begin{proof} (of Theorem~\ref{thmMMNA})
As stated in Lemma~\ref{lemPMMNA}, the Petersen family graphs are all MMNA. What is left is to show that they are the only such graphs on 16 or fewer edges. Suppose $G$ is an MMNA graph with $16$ or fewer edges and suppose that it is not in the Petersen family. If $\delta(G) < 3$, then contracting an edge of a vertex of small degree or deleting an isolated vertex results in a proper minor 
that is still NA, contradicting minor minimality. So we assume $\delta(G) \geq 3$.

Then, since a non-planar graph has at least nine edges,
$G$ must have at least $12$ edges.
If $\|G \| = 12$, it must be cubic. But, then, 
removing a vertex $a$ results in $\| (G-a)^s \| = 6$ so that $G-a$ is planar and $G$ is apex, a contradiction. So we can assume $\|G \| \geq 13$.

Similarly, if $G$ has $13$ edges, then $G$ cannot have a vertex of degree five or more, lest $G-a$ be non-planar. On the other hand, $G$ is certainly not cubic, so,
by Lemma~\ref{lem34mix}, there is a degree $4$ vertex $a$ 
that has a degree 3 neighbor.
Again, $\|(G-a)^s \| \leq 8$, so that $G-a$ is planar, a contradiction. We can assume
$\|G\| \geq 14$.

Suppose $G$ has $14$ edges. 
If $G$ contains a degree 5 vertex $a$, then $G-a$ must be $K_{3,3}$.
By Lemma~\ref{lemapexpath}, $G$ cannot be NA. Suppose there's a degree $4$ vertex $a$ having a degree $3$ neighbor. Then $\| (G-a)^s \| \leq 9$, so $(G-a)^s = K_{3,3}$, as otherwise, $G$ is apex. This also means that $G-a$ is $K_{3,3}$ with a single edge subdivision. By Lemma~\ref{lemda4}, $G$ is not NA, a contradiction.

Having 14 edges, $G$ is not cubic and we've argued that there can be no degree $4$ vertex with a degree $3$ neighbor. So, by Lemma~\ref{lem34mix}, $G$ must be quartic.
Then deleting any vertex $a$ results in a $(6,10)$ graph. If $G$ is NA, $G-a =  K_{3,3}+e$ and since $G$ was $4$-regular, the $N(a)$ is 
exactly the degree three vertices in $K_{3,3}+e$. However, 
$G$ is then apex. We conclude $G$ must have at least $15$ edges.

If $G$ has $15$ edges, then $\Delta(G) \leq 6$ since a non-planar graph has at least nine edges. By Lemma~\ref{lemapathequiv},
if $(G-a)^s$
is $K_{3,3}$ or $K_5$, then $G$ is NA if and only if it has a minor from the Petersen family. Hence, if $G$ is an MMNA $15$ edge graph, finding a vertex whose removal induces a graph that
simplifies to $K_{3,3}$ or $K_5$ implies that $G$ is a member of the Petersen family. 
In particular, if $G$ is cubic (see Lemma~\ref{lemda3}), or has a
a vertex of degree $6$, then $G$ is a member of the Petersen family.

Let us assume that $\Delta(G) = 4$. Since there are no quartic graphs of $15$ edges, by Lemma~\ref{lem34mix}, there is a degree $4$ vertex $a$, with at least one neighbor of degree $3$. 
If $a$ has more than one neighbor
of degree $3$ or 
if $G-a$ is a subdivision of $K_5$ or $K_{3,3}$ then, by Lemma~\ref{lemapathequiv}, we are done. 
In particular, if $a$ has more than one neighbor of degree $3$, then 
$\|(G-a)^s\| \leq 9$. However, as $(G-a)^s$ must be non-planar, $(G-a)^s = K_{3,3}$ and we are done.

So we can assume $a$ has exactly one degree $3$ neighbor and 
that $G-a$ is a subdivision of a $10$ edge non-planar graph other than $K_5$.
Then this graph is either the simple graph $K_{3,3}+e$ or the multigraph
formed by doubling a single edge of $K_{3,3}$ (see Figure~\ref{figK33pe}).

If $G-a$ is the multigraph, then $(G-a)^s = K_{3,3}$ and we can apply Lemma~\ref{lemapathequiv}.
So, suppose $G-a$ is formed by subdividing an edge of $K_{3,3}+e$ (see Figure~\ref{figK33pe}b). If the subdivision is not on the added $v_2v_3$ edge, then $G-w_3$ is planar.  
This is because $a$ is not adjacent to
either $v_3$ or $v_2$ and there is only one additional vertex from subdivision in forming $G-a$. 
So it must be $v_2v_3$ that is subdivided to form $G-a$.
This means $\{w_1, w_2, w_3 \} \subset N(a)$ as otherwise $v_2$ will have two neighbors
of degree $3$. The resulting graph is $K_{4,4}-e$, a member of the Petersen family.

So, we can assume $\Delta(G) = 5$. Let $a$ be a degree $5$ vertex.
Then, being non-planar, $(G-a)^s$ has at least nine edges.
If $(G-a)^s$ is $K_5$ or $K_{3,3}$, Lemma~\ref{lemapathequiv}
implies that we are done. So $G-a$ must be $K_{3,3}+e$ shown in Figure~\ref{figK33pe}b.
Since $a$ is of degree $5$ it must be adjacent to either $w_2$ or $w_3$, say $w_3$. Then, 
$a$ must be adjacent to both $v_3$ and $v_2$, as otherwise $G-w_3$ is planar. However, $a$ must then also be adjacent to $w_1$ and $w_2$. If not,  $v_2$ is a degree $5$ vertex with a degree $3$ neighbor, meaning $G-v_2$ is planar, a contradiction. Thus, $N(a) = \{v_2, v_3, w_1, w_2, w_3\}$ and the resulting 
graph is the $(7,15)$ Petersen family graph that comes from a $\TY$ move on $K_6$. We call this graph $P_7$.

Next suppose $\|G \| = 16$. We can assume $\Delta(G) \leq 6$. Indeed, if
$\Delta(G) \geq 8$, there's a vertex $a$ whose deletion gives $G-a$ of size at most 
eight, hence planar. If $\Delta(G) = 7$, deleting a degree 7 vertex $a$ means
$\|G-a\| = 9$. As $G-a$ must be non-planar, it is $K_{3,3}$ and we can apply Lemma~\ref{lemapathequiv}.

Suppose $\Delta(G) = 6$ and let $a$ be a degree $6$ vertex. Then $G-a$ is a non-planar graph of size 10 and minimal degree at least two. If $G-a$ is $K_5$, we apply Lemma~\ref{lemapathequiv}, so we can assume $G-a$ is $K_{3,3}+e$ (see Figure~\ref{figK33pe}b).  Since $a$ has degree $6$ in $G$, it is adjacent to all vertices of $K_{3,3}+e$ so that $G$ has the Petersen family graph $K_{3,3,1}$ as a subgraph.

If $\Delta(G) = 5$, let $a$ be a vertex of top degree. There are two cases depending on whether or not $a$ has a degree 3 neighbor. If so, $\|(G-a)^s\| \leq 10$. By assumption, $(G-a)^s$ is non-planar and, if $(G-a)^s = K_5$ or $K_{3,3}$, we can
apply Lemma~\ref{lemapathequiv}. So we may assume that $(G-a)^s$ is the graph
$K_{3,3}+e$ (see Figure~\ref{figK33pe}b) and $G-a$ is formed by subdividing a single edge of that graph. If the subdivided edge is the added edge $v_2v_3$, then 
$w_3 \in N(a)$ as otherwise, $G-v_3$ is planar. By symmetry $w_1, w_2 \in N(a)$ as well and $G$ has the Petersen family graph $K_{4,4} -e$ as a subgraph. So, we 
can assume that it is not $v_2v_3$ that is subdivided. 

Suppose it is some other edge incident to $v_2$ or $v_3$, say $v_3w_3$, that is subdivided. Then $G-w_3$ is planar unless $v_2$ and $v_3$ are both neighbors of $a$. But in that case, there will be a degree 5 vertex $b$ with at least two degree 3 neighbors. This means $\|(G-b)^s \| \leq 9$, so it is either planar, a contradiction, or
$K_{3,3}$ and we can apply Lemma~\ref{lemapathequiv}. Thus, the subdivided edge
is adjacent to neither $v_2$ nor $v_3$. Without loss of generality, it is $v_1w_1$ that is split to create $G-a$. Still, $G-w_3$ is planar unless $v_2,v_3 \in N(a)$ and again we will be left with a degree 5 vertex with at least two degree 3 neighbors. 

\begin{figure}[ht]
\begin{center}
\includegraphics[scale=.45]{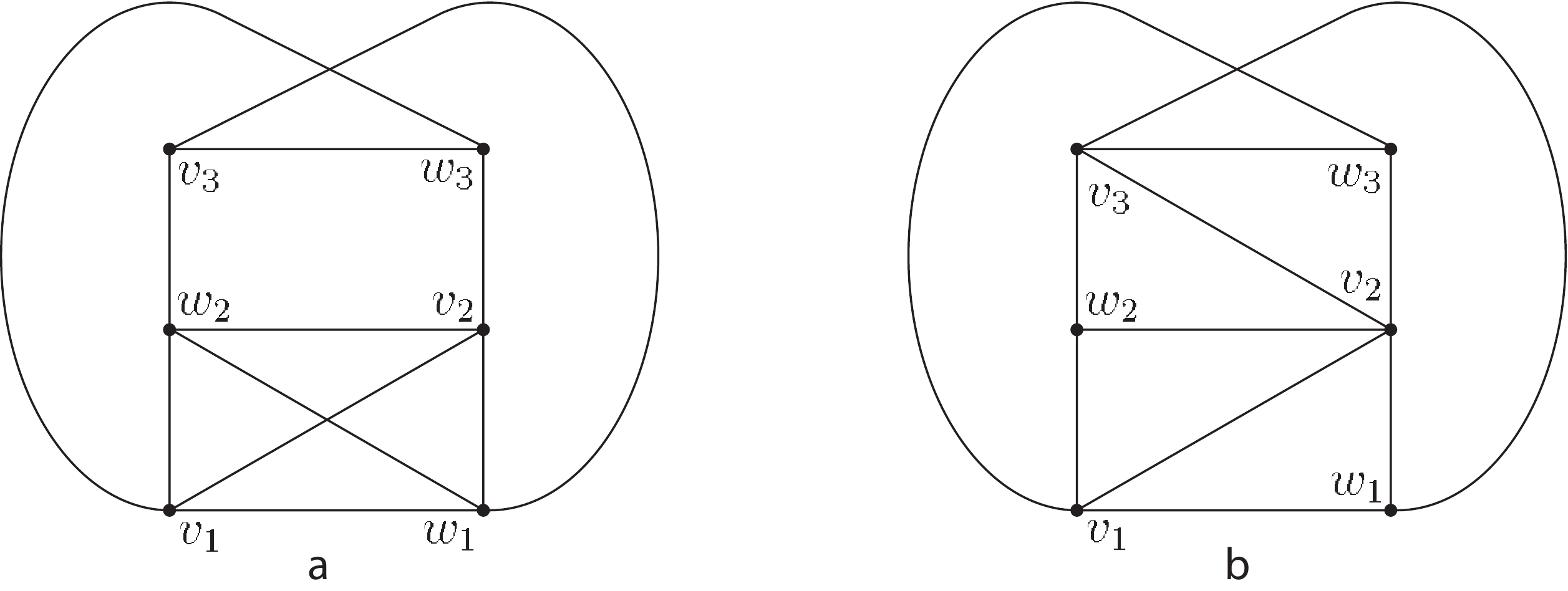}
\caption{Non-planar $(6,11)$ graphs with $\delta(G) \geq 3$.}\label{figNP611}
\end{center}
\end{figure}

So, we can assume $a$ has no degree $3$ neighbor. Then $G-a$ is non-planar, of size 11, and minimal degree three. The only possibilities are the $(6,11)$ graphs of Figure~\ref{figNP611} or the $(7,11)$ graph of Figure~\ref{figNP711}ii. We can assume that no degree 5 vertices have a degree 3 neighbor in $G$ as otherwise we return to the previous case. Suppose first that $G-a$ is the $(6,11)$ graph of Figure~\ref{figNP611}a. Then $N(a)$ must include $v_3$ and $w_3$, the degree 3
vertices of $G-a$ as otherwise there'll be a degree $5$ vertex with a degree $3$ neighbor. Without loss of generality, $w_1$ is the vertex of $G-a$ missing from $N(a)$. Then $G-v_1$ is planar, a contradiction. 
Similarly, if $G-a$ is the $(6,11)$ graph of Figure~\ref{figNP611}b, then, since we assumed $\Delta(G) = 5$, it's $v_2$ that is missing from $N(a)$, in which case $G-w_2$ is planar. 
Finally, suppose $G-a$ is the $(7,11)$ graph of Figure~\ref{figNP711}ii.
We see that $v_2 \in N(a)$ as otherwise, $G-w_3$ is planar. But then $v_2$ is a degree 5 vertex in $G$ and can have no degree 3 neighbors. Thus $N(a) = \{u, v_2, w_1, w_2, w_3 \}$ and contracting $uv_1$ gives the Petersen family graph $P_7$ as a minor. (Recall that $P_7$ is the result of a $\TY$ move on $K_6$.) 

Next assume $\Delta(G) = 4$. If $G$ is quartic, it is one of the six quartic graphs of order eight. Only two of these are NA. One is $K_{4,4}$, which has the Petersen family graph $K_{4,4}-e$ as a subgraph. The other comes from splitting the degree
6 vertex of the Petersen family graph $K_{3,3,1}$. Thus, we can assume $\delta(G) = 3$ and, by Lemma~\ref{lem34mix}, there is a degree $4$ vertex $a$ with a degree $3$ neighbor. Then $\|(G-a)^s\| \leq 11$. By Lemma~\ref{lemapathequiv}, 
$(G-a)^s$ is of size 10 at least, so we can assume each degree $4$ vertex has at most two degree 3 neighbors.

Suppose then that $\| (G-a)^s \| = 10$ meaning $G-a$ is formed by making two edge
subdivisions on $K_{3,3}+e$ (Figure~\ref{figK33pe}b). Suppose further that neither of the subdivisions occur on the added edge $v_2v_3$. Then $G-w_3$ is planar unless the subdivisions are on the edges $v_2w_2$ and $v_3w_2$ (or $v_2w_1$ and $w_3w_1$, a case we can omit due to symmetry.) If these are the subdivisions, then $w_3 \in N(a)$ as otherwise $G-v_3$ is planar. Finally, deleting the vertex on $v_2w_2$ formed by the subdivision, call it $u$, gives a planar graph unless $w_1 \in N(a)$. So, we can assume $a$ is adjacent to $u$, $w_1$, and $w_3$ as well as the vertex formed by subdividing $v_3w_2$. Then, contracting $uw_2$ leads to the $(8,15)$ Petersen family graph resulting from two $\YT$ moves on the Petersen graph. We call this $(8,15)$ graph $P_8$. So, assuming there is no subdivision on $v_2v_3$ leads to a graph with a Petersen family graph minor. 

Thus, we can assume there is at least one subdivision on $v_2v_3$. This means that $v_2$ and $v_3$ already have one degree $3$ neighbor. Since they may have at most two, then two of $w_1$, $w_2$, and $w_3$, say the last two, are adjacent to $a$. In order that $G-w_2$ and $G-w_3$ are both non-planar, the final neighbor of $a$, call it $u$, arises by subdivision of an edge incident to $w_1$. Then contracting $uw_1$ shows that $G$ has the Petersen family graph $K_{4,4}-e$ as a minor.
So, we can assume $\| (G-a)^s \| \geq 11$.

Since $\delta(G)  = 3$, then $|G| \geq 9$. So, if $(G-a)^s $ has size 11, then it has at least order seven. Thus, $(G-a)^s$ is the $(7,11)$ graph of Figure~\ref{figNP711}ii 
and $G-a$ is formed by a single subdivision.  Also, we may assume every degree 4 vertex has at most one degree 3 neighbor (as otherwise we return to the previous case). So that both $G-w_2$ and $G-w_3$ are non-planar, the subdivision must be of $uv_2$ or $v_2w_1$. Either way, this constitutes a degree $3$ neighbor of $v_2$ and its remaining neighbors must all be adjacent to $a$. However, in both cases, this results in a degree $4$ vertex (e.g., $w_1$ or $u$, respectively) with two degree
$3$ neighbors, which puts us back in the previous case. This completes the
argument in the case $\|G \| = 16$ and with it the proof.
\end{proof}

\section{14 vertex graphs}

In this section we show the following (originally proved in \cite{BM}):
\begin{prop} 
If $G$ is a $(14,21)$ MMN2A graph, then $G$ is in the Heawood family.
\end{prop}

\begin{proof} Let $G$ be a $(14, 21)$ MMN2A graph. We can assume $\delta(G) \geq 3$ as otherwise a vertex deletion or
edge contraction on a small degree vertex will give a proper minor that is also N2A.
Then $G$ must have the degree sequence $(3^{14})$ and for any $a \in V(G)$, $G-a$ has the sequence $(3^{10}, 2^3)$. Now choose another vertex, $b$, such that $G^* = G-a,b$ has the sequence $(3^6,2^6)$ (i.e., $a$ and $b$ have no common neighbors). There are enough degree $3$ vertices in $G-a$ to assure we can always choose such a $b$.

Since  $G$ is N2A and $G^*$ has the sequence $(3^6,2^6)$, then $G^*$ must be a split $K_{3,3}$.
By Lemma~\ref{lemda3}, $(G^*+a)^s$ is the Petersen graph.
Then $G' = (G^{\ast}+a) - w_3$ is another split $K_{3,3}$.

\begin{figure}[ht]
\begin{center}
\includegraphics[scale=0.5]{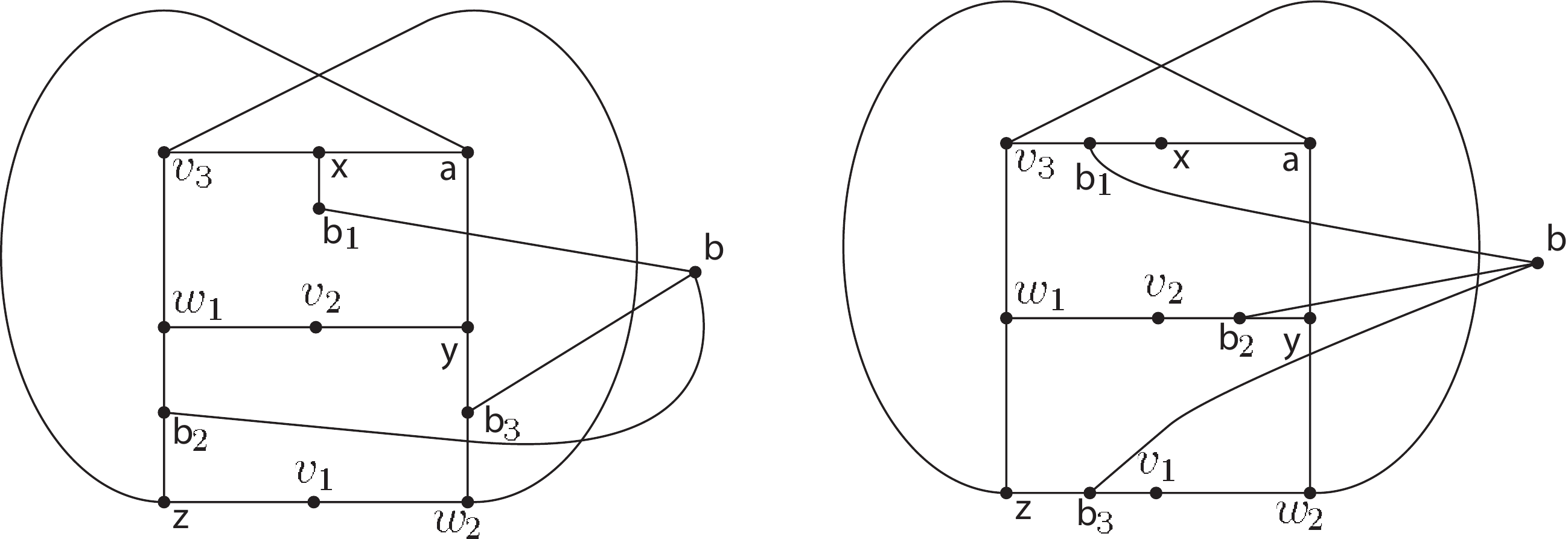}
\caption{Two possibilities for $G^\prime+b$.}\label{figda3andb}
\end{center}
\end{figure}

By  Lemma~\ref{lemapexpath}, $b$ must have a path to $a$ that avoids $v_3$, $w_1$, $w_2$, $y$, and $z$.
Since $a$ and $b$ have no common neighbors, this means $b$ has a neighbor $b_1$ that is adjacent
to $x$. So, there are two cases: in $G'+b$, either $b_1$ is of degree two, or else it has $v_3$ as a third neighbor. (See Figure~\ref{figda3andb}.)

In either case, $b_1$ gives paths from $b$ to the  branch vertices $a$ and $v_3$ and there are three ways to split the remaining four  branch vertices into two pairs. However, we see that $G- w_2,z$ is planar 
(and $G$ is $2$--apex), unless we make the choices shown in Figure~\ref{figda3andb}.
In both cases, adding $w_3$ back will give us the Heawood graph.
Hence the only (14,21) MMN2A graph is the Heawood graph, which is in the
Heawood family.
\end{proof}

\section{13 vertex graphs}

In this section we prove the following:
\begin{prop} If $G$ is a $(13, 21)$ MMN2A graph, then $G$ is in the Heawood family. 
\end{prop}

\begin{proof}
Let $G$ be a MMN2A $(13,21)$ graph.
Consider the degree sequences $(3^{12},6)$ and $(3^{11},4,5)$. If we remove the vertex of highest degree the resulting graph simplifies to a graph with fewer than $14$ edges, hence (by Theorem~\ref{thmMMNA}) to an apex graph. So
$G$ does not have such a degree sequence.

Then $G$ has the sequence $(3^{10},4^3)$. Again,
if $a$  is a vertex of degree $4$ that has three neighbors of degree $3$, then $(G-a)^s$ is apex, so this cannot be the case. We conclude
that the degree $4$ vertices form a triangle in $G$ and that there is a degree $3$ vertex in $G$, call it $a$, whose neighbors all have degree $3$.
This means that $G-a$ simplifies to a graph $G^* = (G-a)^s$ with degree sequence $(3^{6},4^3)$. Since $G^*$ must be NA, and has $15$ edges, by Theorem~\ref{thmMMNA} it is in the Petersen family.  There is a unique nine vertex graph in the family, which we call $P_9$, see Figure~\ref{figP9}.

\begin{figure}[ht]
\begin{center}
\includegraphics[scale=0.5]{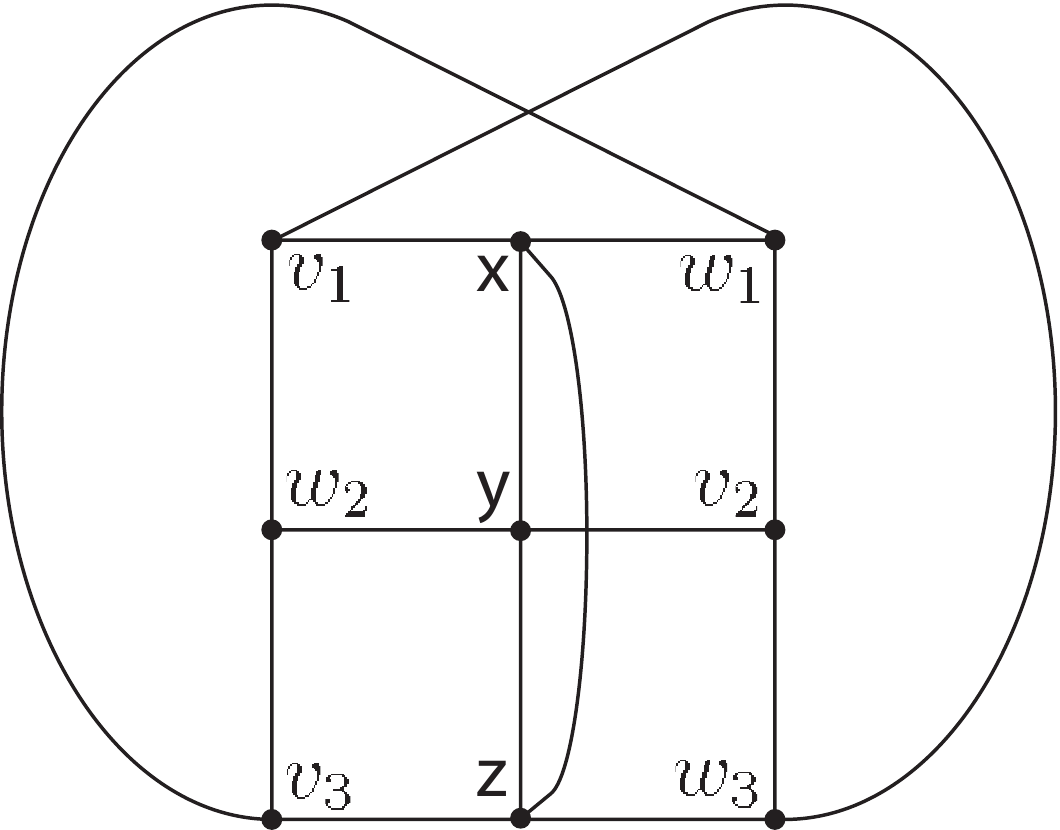}
\caption{The Petersen family graph $P_9$.}\label{figP9}
\end{center}
\end{figure}

Note that in Figure~\ref{figP9} there is a unique triangle, which we'll denote $xyz$
and label the corresponding vertices in $G-a$ 
and $G$ as $x$, $y$, and $z$ as well.  
Notice also that $x$, $y$ and $z$ all have degree $4$ in $G^*$ so none of them are neighbors of $a$ in $G$. Moreover, we assumed $x$, $y$ and $z$ form a triangle in $G$, and since the triangle is clearly preserved
in $G^*$, it must also be preserved in $G-a$. In particular, this implies that $a$ is not near any of the edges that form this triangle, i.e., none of the degree $2$ vertices deleted in simplifying from $G-a$ to $G^*$ are on the edges of the
triangle. 

Observe that $(G-a,y)^s = K_{3,3}$ and that the induced graph after adding $a$ back must be NA. Hence, by Lemma~\ref{lemapexpath}, $a$ must have a path to each  branch vertex that does not go through any other  branch vertex. Since $a$ is not near the edge $xz$, it must be near either edges $xw_1$ or $xv_1$ and $zw_3$ or $zv_3$. Similarly, $(G-a,x)^s$ shows that $a$ must also be near
$yw_2$ or $yv_2$. 

We claim that $a$ is near $xw_1$, $yw_2$, and $zw_3$ or $xv_1$, $yv_2$, and $zv_3$, in which case $G$ is the Heawood family graph $C_{13}$. (See \cite{HNTY} for the names, like $C_{13}$, of the Heawood family graphs. This is the unique order 13 graph in the Heawood family and corresponds to graph 15 in Figure~\ref{figHea}). Otherwise, either $a$ is near $xv_1$ and $yw_2$ or $xw_1$ and $yv_2$, in which case 
$G-v_3,w_3$ is planar, or else $a$ is near $zv_3$ and $yw_2$ or $zw_3$ and $yv_2$ in which case $G-v_1,w_1$ is planar. Therefore the proposition is proved.
\end{proof}

\section{12 vertex graphs}

In this section we prove that a $(12,21)$ MMN2A graph $G$ is in the Heawood family. This means $G$ is one of three graphs that are called $H_{12}$, $C_{12}$, and $N'_{12}$ by Hanaki et al.~\cite{HNTY} and are represented as graphs 12, 13, and 19, respectively, in Figure~\ref{figHea}. We first observe that if $G$ is triangle-free
and of the correct degree sequence, it must be $H_{12}$. This was originally proved in \cite{BM}.

\begin{lemma} \label{lemtfreeH12}%
Let $G$ be MMN2A of degree sequence $(3^6,4^6)$ and triangle free. Then $G$ is $H_{12}$.
\end{lemma}

\begin{proof} 
Note that if any of the vertices of degree $4$ have three or more neighbors of degree $3$, removing such a vertex results in an apex graph by 
Theorem~\ref{thmMMNA}, so we may assume this doesn't happen.
We also notice that we can either single out a degree 3 vertex, all of whose neighbors are degree 3 vertices, or a degree 4 vertex that has two degree 3 neighbors.
To see this, suppose it is not the case. Since $G$ has no triangles, the subgraph induced by the degree $4$ vertices is $K_{3,3}$ and each of the vertices has a unique neighbor of degree $3$. 
Hence, removing two non-adjacent vertices of degree $4$ results in a graph that simplifies to a graph of size eight, thus planar. Hence $G$ would not be $2$-apex.

Now assume that we do not have a vertex of degree $4$ with two degree $3$ neighbors. Say that $a$ is a degree $3$ vertex whose neighbors are all of degree $3$. 
Then $(G-a)^s$ has degree sequence $(3^2,4^6)$. Theorem~\ref{thmMMNA} implies that it is $K_{4,4}-e$. Because $G$ has no degree 4 vertex with two degree $3$ neighbors, we know that the edge subdivisions
from $(G-a)^s$ to $G-a$ are all on edges incident to the degree 3 vertices of $(G-a)^s$. Since there are exactly three subdivisions from $(G-a)^s$ to $G-a$, there is one vertex of degree $3$ in $(G-a)^s$ that
gets at least two subdivisions, call it $a_1$. So, $a_1$ has degree 4 neighbors 
$v_1$, $v_2$ in $(G-a)^s$ so that $a_1v_1$ and $a_2v_2$ are subdivided in
forming $(G-a)$.
Then $G-v_1,v_2$ is planar and $G$ is $2$-apex.

So we may assume that $a$ has degree $4$ and there exist $b,c\in N(a)$ such that $d(b)=d(c)=3$ and $c\neq b$. Then $(G-a)^s$ has degree sequence $(3^6,4^3)$ which tells us, by Theorem~\ref{thmMMNA}, that it is
$P_9$. Furthermore, since $G$ does not have a triangle, we know that one of the subdivisions from $(G-a)^s$ to $G-a$ is on the triangle $xyz$ of Figure~\ref{figP9};
say it's $xy$ that is subdivided.
Removing either $x$ or $y$, Lemma~\ref{lemapexpath} tells us that the other subdivision from $(G-a)^s$ to $G-a$ must be on an edge incident to $z$. We may say it is the edge $yz$ without losing generality. Now, remove $y$
and it is easy to see that Lemma~\ref{lemapexpath} forces $a$ to be adjacent to $w_2$ and $v_2$. Therefore $G$ is $H_{12}$.
 \end{proof}
 
\begin{prop} If $G$ is a $(12,21)$ MMN2A graph, then $G$ is in the Heawood family.
\end{prop}

\begin{proof}
We assume again that $G$ is MMN2A and that $G$ is a $(12,21)$ graph. We 
can assume the maximum degree $\Delta(G)$ is at most five.  For a vertex $a$ with $d(a) \geq 6$ in a $(12,21)$ graph with $\delta(G) \geq 3$
will have at least one neighbor of degree $3$. Then $(G-a)^s$ has at most 14 edges
and is apex, by Theorem~\ref{thmMMNA}. This implies $G-a$ is apex and $G$ is $2$--apex, a contradiction.

This leaves four possible degree sequences:
$(3^9,5^3)$, $(3^8,4^2,5^2)$, $(3^7,4^4,5)$, and $(3^6,4^6)$. 

Let $G$ have the degree sequence $(3^9,5^3)$ or $(3^8,4^2,5^2)$. Then any $a$ with $d(a) = 5$ has at least two neighbors of
degree $3$. This means $(G-a)^s$ simplifies to a graph with fewer than $15$ edges and so it is apex (Theorem~\ref{thmMMNA}), whence $G$ is $2$--apex, a contradiction.

\begin{figure}[ht]
\begin{center}
\includegraphics[scale=0.5]{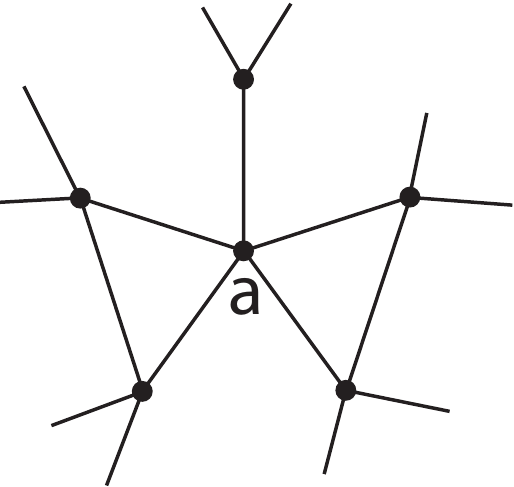}
\caption{Graph near the degree 5 vertex $a$.}\label{figC12start}
\end{center}
\end{figure}

We now focus our attention on the case where $G$ has the degree sequence $(3^7,4^4,5)$ and show that the only MMN2A graph with this degree sequence is $C_{12}$. 
(See \cite{HNTY} for the name. This is graph 12 in Figure~\ref{figHea}.)
Let $a$ denote the vertex of degree $5$. Note that $a$ has at most one neighbor of degree $3$, as otherwise
$\|(G-a)^s\| \leq 14$ meaning $G-a$ is apex (Theorem~\ref{thmMMNA}) and $G$ is $2$--apex.
Hence, the neighbors of $a$ are all the vertices of degree $4$ and one vertex of degree $3$. Moreover, each vertex of degree $4$ has at most $2$ neighbors
of degree $3$. This is illustrated in Figure~\ref{figC12start} . This implies that $(G-a)^s$ is a NA $3$-regular graph with $15$ edges, i.e., the Petersen graph (see Figure~\ref{figda3}). 
Since the Petersen graph has no triangles or $4$-cycles, we see that $G-a$ has no four cycles. This implies that the vertices of degree $4$
do not form a triangle or $4$-cycle in $G$. This justifies the specifics of Figure~\ref{figC12start}.

Let $b\in V(G)$ denote one of the vertices of degree $4$. From the above paragraph, we argued that $b$ must have exactly two neighbors of degree $3$, hence $(G-b)^s$ is a $(9,15)$ graph with degree sequence $(3^6,4^3)$.
This implies that $(G-b)^s$ is the Petersen family graph $P_9$ illustrated in Figure~\ref{figP9} (the unique Petersen family graph on nine vertices).
In $G-b$, vertex $a$ has degree 4 and without loss of generality is vertex $y$ in the figure. We have deduced
that $b$ is adjacent to $a$ as well as either $w_2$ or $v_2$, say $v_2$. Note that $b$ is not near the edge $xz$. In order for $G-a$ to be NA, by Lemma~\ref{lemapexpath}, $b$ is near the edges $v_1x$ and $v_3z$. Adding both $a$ and $b$ back in shows that this graph is $C_{12}$.

Now let $G$ have the degree sequence $(3^6,4^6)$. We will show $G$ is either $H_{12}$ or else $N'_{12}$.
(See \cite{HNTY} for these names. There are graphs 12 and 19 respectively in Figure~\ref{figHea}.)
 By Lemma~\ref{lemtfreeH12}, the only triangle free MMN2A graph with degree sequence $(3^6,4^6)$
is $H_{12}$, so we will assume that $G$ has a triangle and show that this implies it is $N'_{12}$. By Theorem~\ref{thmMMNA}, each degree four vertex in $G$ can have at most two neighbors of degree $3$. Notice that in $N'_{12}$, each degree $4$ vertex has exactly one neighbor of degree $3$ and vice versa.
We argue that $G$  must also share this property in order to be MMN2A.

\begin{figure}[ht]
\begin{center}
\includegraphics[scale=0.5]{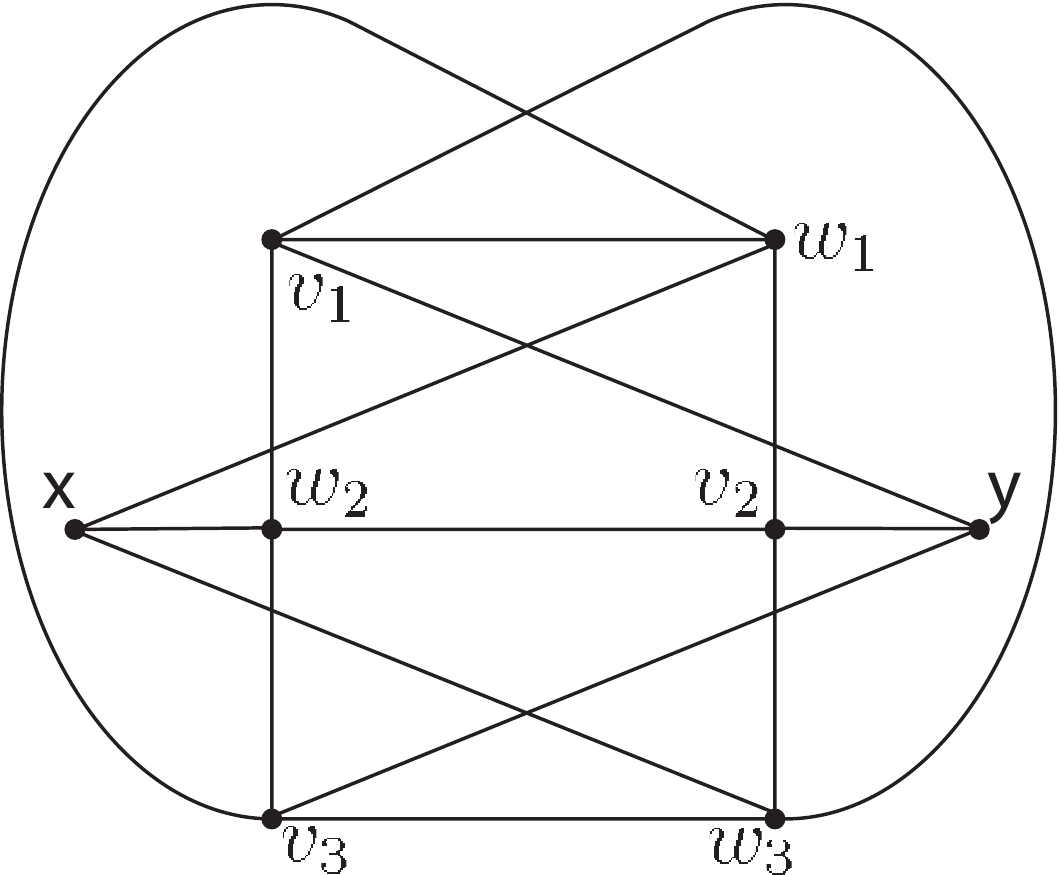}
\caption{The Petersen family graph $K_{4,4}-e$.}\label{figK44me}
\end{center}
\end{figure}

First, assume there is an $a\in V(G)$ such that $a$ has degree $3$ and three degree $3$ neighbors. Hence $G^* = (G-a)^s$ has 
degree sequence $(4^6,3^2)$ and is an $(8,15)$ graph. Since $G$ being
MMN2A implies that $G^*$ is NA, by Theorem~\ref{thmMMNA} it is in the Petersen family. By the degree sequence $(4^6,3^2)$, we can identify $G^*$ as $K_{4,4}-e$ drawn in Figure~\ref{figK44me}.  Since $G^*$ has no triangles, the triangle 
of $G$ is formed in reattaching $a$. Hence there is at least one edge in $G^*$ that
is subdivided twice in returning to $G-a$. Because of the symmetry of $G^*$,
we may assume without loss of generality that these subdivisions are on the edges $v_1w_1$ or $yv_1$. In the first case $G-v_1,w_1$ is planar and the second splits into two cases: either the other subdivsion from $G^*$ to $G-a$
occurs on an edge incident to $x$ in $G^*$ or it does not. In the case where it does not, then $G-v_i,w_j$ is planar, where $v_i$ and $w_j$ are the vertices in $G^*$ 
between which the subdivision occurs or $v_1 w_1$ if it's on an edge incident to $y$. In the other case, $G-x,v_1$ is planar
since it is essentially the same as the planar graph $G^*-x,v_1$ with an extra path from $y$ to a $w_i$ . So, in an MMN2A graph, every degree 3 vertex has at least one degree 4 neighbor.

Now suppose $a\in V(G)$ is a degree 4 vertex with exactly two neighbors of degree $3$. Then $G^* = (G-a)^s$ has degree sequence $(4^3,3^6)$. Since $G^*$ must be NA, by Theorem~\ref{thmMMNA} it is 
in the Petersen family and hence is the graph $P_9$ shown in Figure~\ref{figP9}. In the following, we use the labeling of that figure.

When we remove $x$, $y$, or $z$ separately from $G^*$ each induced subgraph shows us (by Lemma~\ref{lemapexpath}) that $a$ must have paths to $x$, $y$ and $z$ in $G$ that do not include any of their
neighbors in $G^*$. As these three vertices already have degree 4, the neighborhood of $a$ includes vertices adjacent to $x$, $y$, $z$ created by edge subdivisions.

Since there are only two edge subdivisions from $G^*$ to $G-a$, this implies that one has to be on the $xyz$ triangle. By the symmetry of $G^*$ we can assume
without loss of generality that $xy$ is subdivided. The other subdivision is on an edge incident to $z$ in $G^*$. Since we assume that $G$ contains a triangle, $a$ must be part of that triangle. 
Observe that $(G^*-y)^s = K_{3,3}$. By Lemma~\ref{lemapexpath}, $a$ must have paths
to the vertices $v_1$ $v_3$, $w_1$, $w_3$, $x$, and $z$ in $G-y$ that exclude the others from that list. Now, $a$ is adjacent to exactly two vertices in $G^*-y$ (as the two other neighbors appear only after additional edge subdivisions)  and since we have already established that $a$ is near both
$x$ and $z$ and possibly $v_3$ or $w_3$,
the remaining neighbors of $a$  are either $w_2$ and $v_2$, $v_1$ and $v_2$, or $w_1$ and $w_2$. Recalling that $a$ is not actually adjacent to $x$, just simply near it by way of a subdivison of $xy$ in $G^*$, and since
$G$ must have a triangle, none of these cases can be $G$.

\begin{figure}[ht]
\begin{center}
\includegraphics[scale=0.5]{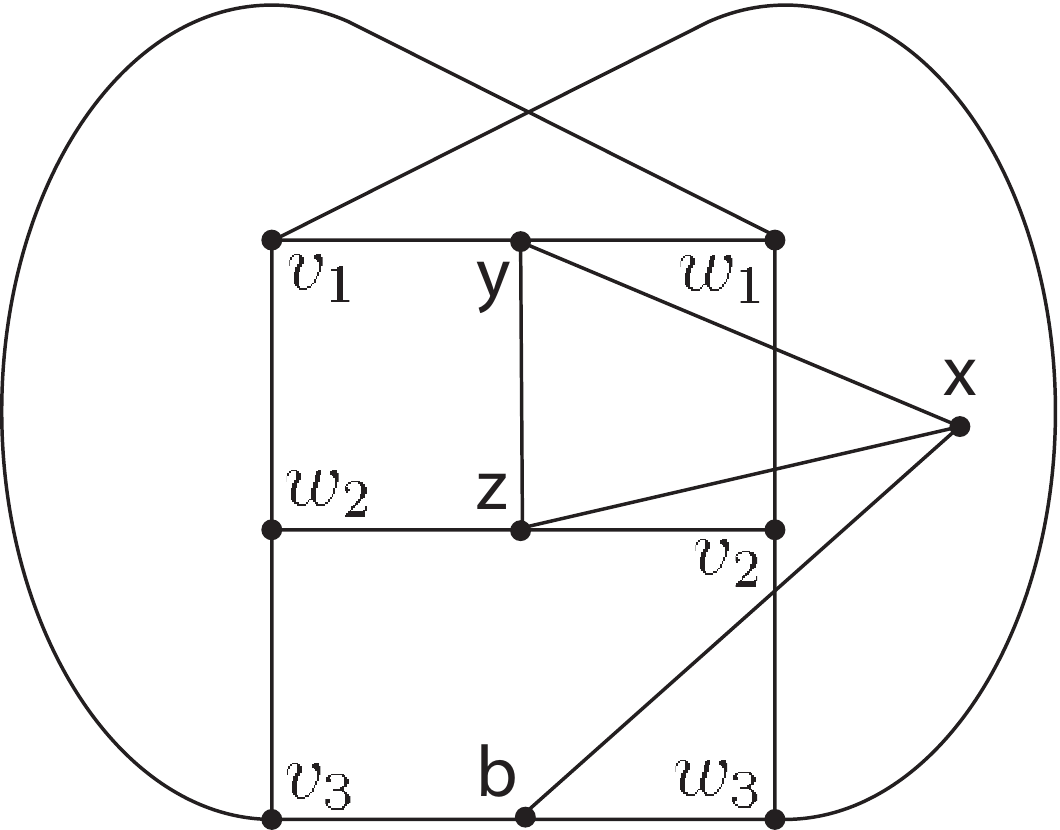}
\caption{Graph after removing a degree $4$ vertex leaving a triangle.}\label{figda4wtri}
\end{center}
\end{figure}

To summarize, we established that if $G$ is MMN2A with degree sequence $(3^6,4^6)$ and contains a triangle, then each vertex of degree $4$ has at most one neighbor of degree $3$ and each vertex of degree $3$ has at least 
one neighbor of degree $4$. Hence, there is a one to one correspondence between the degree $4$ vertices and the degree $3$ vertices by the relation of being neighbors in $G$. Note that none of the degree $3$ vertices can be
part of a triangle in $G$, otherwise, it would either be adjacent to at least two degree 4 vertices or else there is a degree $4$ vertex with two neighbors of degree $3$. Thus, we can assume there is a triangle of
vertices of degree $4$ in $G$. Choose some vertex of degree $4$ not on this triangle, call it $a$. Then $G^* = (G-a)^s$ has degree sequence $(3^8, 4^2)$ and contains a triangle.   
We claim that $G^*$ is the graph illustrated in Figure~\ref{figda4wtri}. Note that 
the two degree $4$ vertices in $G^*$ are adjacent. So,  
if we delete one of them, denote it $b$, then $(G^*-b)^s$ has nine edges and 
must be non-planar since $G^*$ is NA. 
Thus $(G^*-b)^s = K_{3,3}$ and, using Lemma~\ref{lemda4}, and that $G^*$ has a triangle and degree sequence $(3^8, 4^2)$, we deduce  $G^*$ is as shown in Figure~\ref{figda4wtri}. 

Now that we have established what $G^*$ looks like (Figure~\ref{figda4wtri}), determining where $a$ goes is easy. For starters, since both $y$ and $z$ are adjacent to $x$, then $x$ cannot have degree $3$ due to the one to one correspondence between 
vertices of degree 3 and 4. So $a$ is adjacent to $x$.
Either $a$ is adjacent to $v_1$ or $w_1$ since $y$ is adjacent to only one vertex of degree $3$, say $w_1$. Then, for the same reason $x$ and $a$ were adjacent, $a$ and $v_2$ are adjacent. Since $G-z$ is NA, by Lemma~\ref{lemapexpath},
$a$ is near $w_2v_3$ or $v_1w_2$. Similarly, $G-y$ is NA and Lemma~\ref{lemapexpath} shows $a$ is near $v_1w_2$ or $v_1w_3$. So $a$ is near $v_1w_2$.
This graph is $N'_{12}$. Therefore, the only graph MMN2A graph with degree sequence $(3^6,4^6)$ that contains a triangle is $N'_{12}$.
\end{proof}

\section{11 vertex graphs}

In this section we prove that an $(11,21)$ MMN2A graph is in the Heawood family. 
We begin with five lemmas, one each for  the Heawood family graphs of this order: $E_{11}$, $C_{11}$, $H_{11}$, $N'_{11}$, and $N_{11}$. (See \cite{HNTY} for the names. These correspond to graphs  8, 10, 11, 16, and 17 respectively in Figure~\ref{figHea}.)

\begin{lemma}  Let $G$ be an $(11,21)$ MMN2A graph with degree sequence $(3^4,4^6,6)$. Then $G$ is $C_{11}$.
\end{lemma}

\begin{proof}
Consider $b\in V(G)$ such that $\deg(b)=6$. Notice that for any $v\in N(b)$ we must have $\deg(v)=4$, otherwise (Theorem~\ref{thmMMNA}) $G-b$ is not NA. This implies that $G-b$ must be the Petersen graph (see Figure~\ref{figda3}). Without loss of generality,
we can assume that the vertex $a$ in Figure~\ref{figda3} is not a neighbor of $b$ in $G$. Since $(G-b,x)^s  = K_{3,3}$, then in $G-x$, by Lemma~\ref{lemapexpath}, $b$ must be adjacent to $z$ and $y$. Similarly, if we consider
$G-b,z$ we see that $b$ is adjacent to $x$. Consider again $G-x$. Since $b$ has degree $5$ in $G-x$, is adjacent to $y$ and $z$, and must have paths to $v_1$, $v_2$, $w_1$, and $w_2$ that do not go through 
$v_1$, $v_2$, $w_1$, $w_2$, $x$, or $y$, we see that $b$ is adjacent to either $v_3$ or $w_2$ or both. Similarly, considering $G-y$ and $G-z$, we see that $b$ is adjacent to either $v_2$ or $w_2$ and $v_1$ or $w_1$.
We claim that $b$ is adjacent to $v_1,v_2$, and $v_3$ or $w_1,w_2$, and $w_3$ in which case we have $C_{11}$. Otherwise, if $v_2\in N(b)$ and $w_1\in N(b)$ then $G-v_3,w_3$ is planar,
or if $v_2\in N(b)$ and $w_3\in N(b)$ then $G-w_1,v_1$ is planar. Similarly, if $w_2\in N(b)$ and $v_1\in N(b)$ then $G-v_3,w_3$ is planar, or if $w_2\in N(b)$ and $v_3\in N(b)$ then $G-w_1,v_1$ is planar. 
Therefore $G$ must be $C_{11}$.
\end{proof}

\begin{lemma} Let $G$ be an $(11,21)$ MMN2A graph with degree sequence $(3^5,4^3,5^3)$. Then $G$ is $E_{11}$.
\end{lemma}

\begin{proof}
We may
assume that $\exists a\in V(G)$ such that $\deg(a)=5$ and $\exists u\in N(a)$ such that $\deg(u)=3$. If not, then removing any two of the degree $4$ vertices results in a $K_4$ graph with a bridge to a graph of at most seven edges,
which is clearly planar. So we may assume that $G^* = (G-a)^s$ has degree sequence $(3^6,4^3)$. This means that $G^*$ is the Petersen family graph $P_9$ shown in Figure~\ref{figP9}. By the
degree sequence of the original $G$, we may assume, without loss of generality, that $a$ is adjacent to $x$ and $y$ (referring again to Figure~\ref{figP9}), and hence is not adjacent to $z$. Removing either $x$ or $y$, Lemma~\ref{lemapexpath}
shows us that $a$ is near an edge incident to $z$. If $a$ is near the edge $yz$ or $xz$, then $a$ is also adjacent to two more vertices in Figure~\ref{figP9}. Removing both of these results in a planar graph. Thus $a$ is near the edge $v_3z$ 
or the edge $w_3z$. By symmetry, we will assume $v_3z$.

Applying Lemma~\ref{lemapexpath} to $G-y$ shows that $a$ must be adjacent to $v_2$ and, similarly, considering $G-x$ shows us that $a$ must be adjacent to $v_1$. Reassembling $G$ gives $E_{11}$.
\end{proof}

\begin{lemma}
Let $G$ be an $(11,21)$ MMN2A graph with degree sequence $(3^4,4^5,5^2)$. Then $G$ is $H_{11}$.
\end{lemma}

\begin{proof}
Assume that $\exists a\in V(G)$ such that $\deg(a)=5$ and $\exists u\in N(a)$ such that $\deg(u)=3$. Then $G^* = (G-a)^s$ is a $(9,15)$ NA graph, hence the graph illustrated in Figure~\ref{figP9}, with degree sequence $(3^6,4^3)$.
Since $G$ has only two vertices of degree $5$, vertex $a$ is adjacent to at most one of $x$, $y$, and $z$ in Figure~\ref{figP9}. We will assume that it is $x$ and hence $y,z \notin N(a)$. By Lemma~\ref{lemapexpath}, $a$
must be near edges incident to both $y$ and $z$ (consider $G-z$ and $G-y$, respectively). However, as $a$ has a unique neighbor of degree $2$ in $G-a$, it is  near only one edge. Therefore, $a$ is near the edge $yz$. If $a$ is adjacent to $v_1$, $v_2$, and $v_3$ or 
$w_1$, $w_2$, and $w_3$ then $G$ is $H_{11}$.

We next verify that this must be the case. Note that there are exactly three vertices
in $N(a) \cup \{v_1, v_2, v_3, w_1, w_2, w_3 \}$. Let us first examine the intersection
with $\{v_2, v_3, w_2, w_3 \}$. Lemma~\ref{lemapexpath} applied to $G-z$ shows that $a$ has 
at least one neighbor in each of the pairs $\{v_2, w_3\}$, $\{v_3, w_2\}$, and $\{v_3, w_3 \}$. The same lemma with $G-x$ shows that $N(a) \cap \{v_2,v_3,w_2,w_3\}$ is not simply $\{v_3, w_3 \}$. We conclude that $a$ is adjacent to $w_2$ and $w_3$ or $v_2$ and $v_3$, and, by symmetry, we can assume $v_2$ and $v_3$.
The last neighbor of $a$ must be $v_1$, as otherwise $G-v_3,w_3$ or $G-v_2,w_2$ will be planar.

\begin{figure}[ht]
\begin{center}
\includegraphics[scale=0.5]{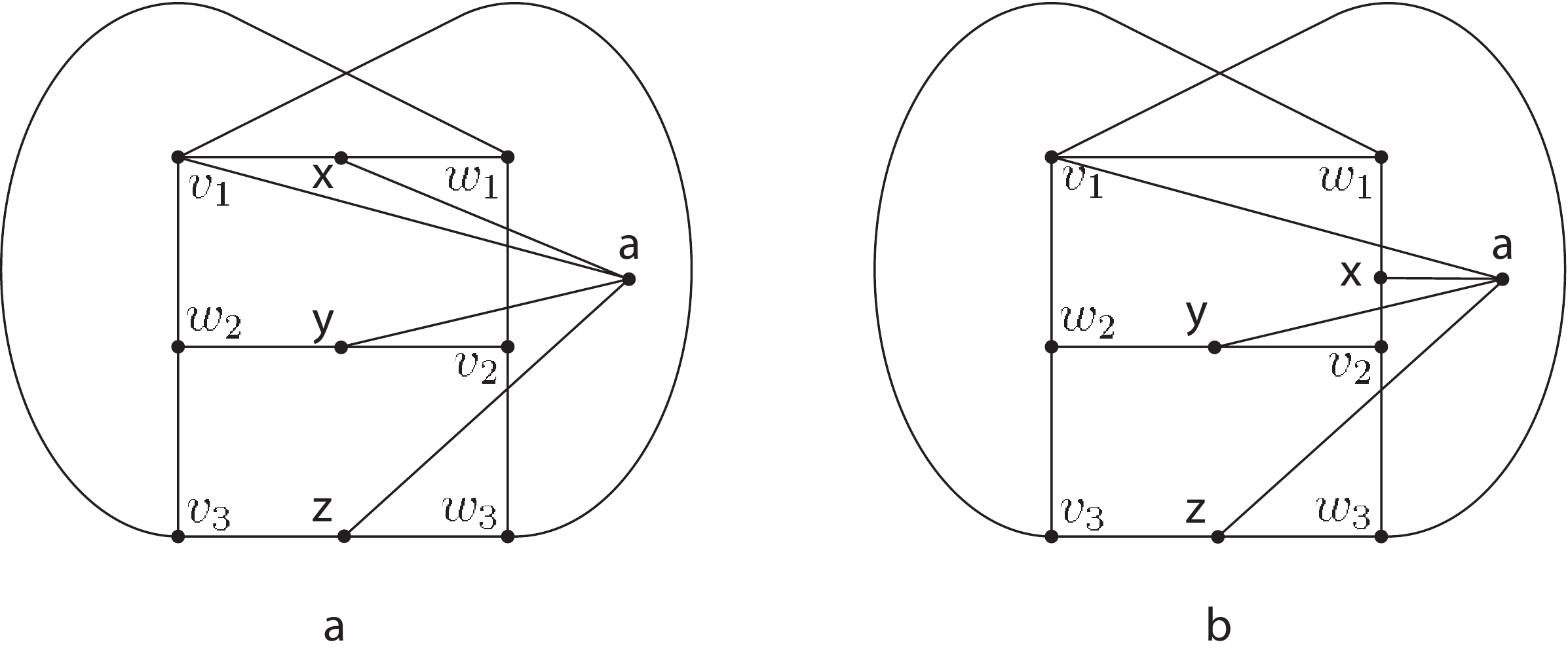}
\caption{Graphs with degree sequence $(3^8, 4^2)$ by adding a degree 4 vertex $a$ to a split $K_{3,3}$.}\label{figH11lem}
\end{center}
\end{figure}

Let $a$ and $b$ be the degree 5 vertices and suppose neither has a degree $3$ neighbor. If $a$ and $b$ are not adjacent, then $(\Gab)^s$ is 
a $(3^4)$ graph that is clearly 
planar. Further, $a$ and $b$ can have at most three common neighbors, as otherwise $(\Gab)^s$ has fewer than nine edges and is therefore planar. On the other hand, since there are only five degree 4 vertices, $a$ and $b$ must share at least three neighbors. This means $(\Gab)^s = K_{3,3}$.
By Lemma~\ref{lemda4}, $G-b$ must be as in Figure~\ref{figH11lem}a or b.
By our assumption, $b$ is adjacent to $a$, $x$, $y,$ and $z$, with one other neighbor from the set $\{w_1,w_2,w_3,v_2,v_3\}$. In the case where $G-b$ looks like Figure~\ref{figH11lem}a we see that $G-v_1,w_1$ is planar.
For the case of graph b in the figure, observe that $G-v_1,x$ is planar. Hence if $a$ and $b$ have no degree 3 neighbors, then $G$ is $2$--apex. Therefore $G$ must be $H_{11}$.
\end{proof}

\begin{lemma}
Let $G$ be an $(11,21)$ MMN2A graph with degree sequence $(3^3,4^7,5)$. Then $G$ is $N'_{11}$.
\end{lemma}

\begin{proof}
Let us begin by assuming that the degree $5$ vertex, call it $b$, is adjacent to some vertex of degree $3$. Then $G^* = (G-b)^s$ has degree sequence $(3^6,4^3)$
and is therefore the $P_9$ graph of Figure~\ref{figP9}. 
Note that $b$ is not adjacent to $x$, $y$, or $z$, since going from $G$ to $G^*$ did not change their degree. 
However, observing the graphs we obtain when removing $x$, $y,$ or $z$, by Lemma~\ref{lemapexpath} we see that $b$ needs a path to all of them that does not utilize any of their neighbors in $G^*$. This is clearly impossible
since there is at most one subdivision from $G^*$ to $G-b$. Hence $\forall v\in N(b)$ we have $\deg(v)=4$.

Then $G-b$ must have the degree sequence $(3^8,4^2)$. If the vertices of degree $4$ in $G-b$ are not adjacent, then if $v$ is one of those, $(G-b,v)^s$ has eight edges and is therefore planar, which is a contradiction.
So choose $a\in V(G-b)$ such that $\deg(a)=4$. Then if $G$ is N2A, $(\Gab)^s$ is $K_{3,3}$. When we add $a$ back in, by Lemma~\ref{lemda4}, there are two cases, shown in Figure~\ref{figH11lem}.
However, for Figure~\ref{figH11lem}b, we notice that $b$ is not adjacent to $v_1$
since it can only be adjacent to vertices of degree $3$ in $G-b$. This means that it is not near $v_1$ which is required by Lemma~\ref{lemapexpath}. So $G-b$ is isomorphic to the graph illustrated in Figure~\ref{figH11lem}a. As above, since $b$ must be near $v_1$, it must be adjacent to $x$. Now, $G-v_1,w_1$ will be planar unless $N(b)$ includes either $\{v_2,v_3\}$ or $\{w_2,w_3\}$. We will argue that
it must be the latter. Suppose instead $\{x,v_2,v_3 \}$ is in $N(b)$ and $\{w_2,w_3\}$
is not. In particular, if $w_2 \notin N(b)$, then $G-v_3,w_3$ is planar, a contradiction. Similarly, if $w_3 \notin N(b)$, $G - v_2,w_2$ gives a contradiction. This shows that it is not possible that $\{w_2,w_3\} \not\subset N(b)$, and so we can assume
$\{w_2, w_3 \} \subset N(b)$. Now $G-v_2,w_2$ is planar unless $b$ is adjacent to $y$ and $G-v_3,w_3$ shows $z$ is adjacent to $b$ as well, which means $G$ is $N'_{11}$.
\end{proof}

\begin{lemma} Let $G$ be an $(11,21)$ MMN2A graph with degree sequence $(3^2,4^9)$. Then $G$ is $N_{11}$.
\end{lemma}

\begin{figure}[ht]
\begin{center}
\includegraphics[scale=0.5]{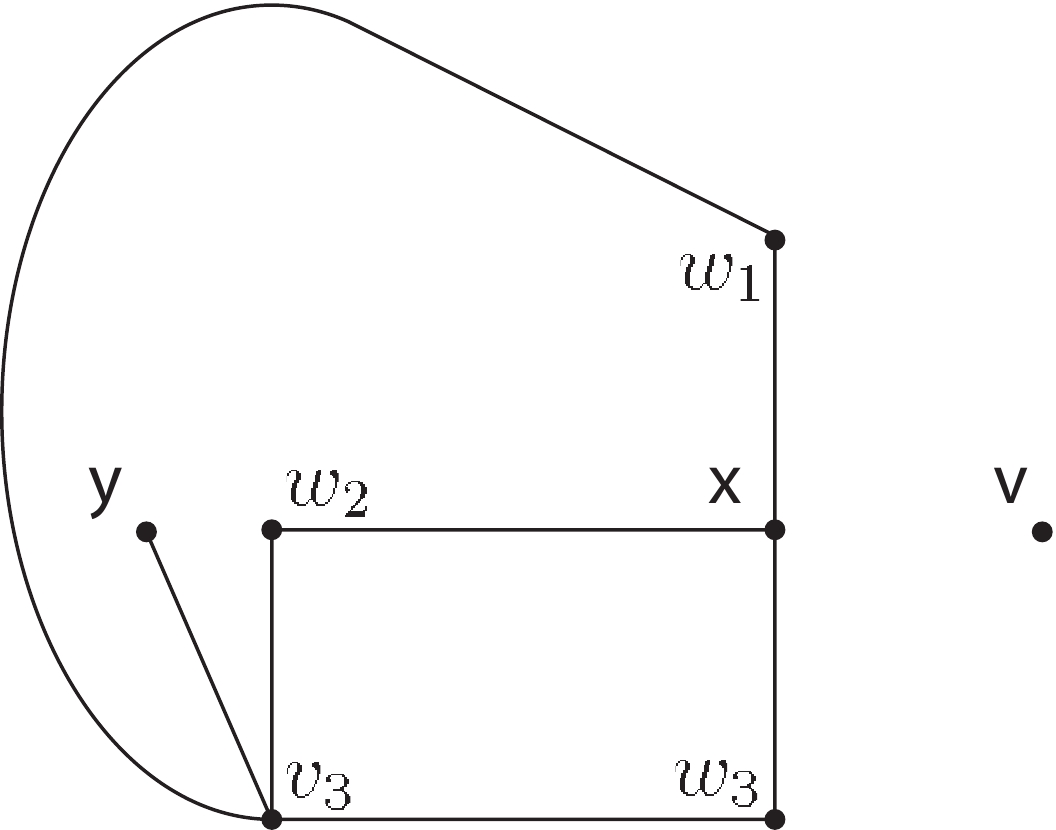}
\caption{Remove $v_1$ and $v_2$ from $K_{4,4}-e$.}\label{figN11lem}
\end{center}
\end{figure}

\begin{proof}
First assume that there exists a $v\in V(G)$ such that $\deg(v)=4$ and the two vertices of degree $3$ are neighbors of $v$. Then $(G-v)^s$ has degree sequence $(3^2,4^6)$ and is the Petersen family graph $K_{4,4}-e$
illustrated in Figure~\ref{figK44me}. Thus $G-v$ is a subdivision of $K_{4,4}-e$.
Note that in $G$, vertex $v$ is adjacent to both $x$ and $y$. The graph obtained from $K_{4,4} -e$ when we remove 
$v_1$ and $v_2$ is illustrated in Figure~\ref{figN11lem}. Since $v$ is adjacent to both $x$ and $y$ and the graph $G-v,v_1, v_2$ can be obtained from Figure~\ref{figN11lem} by only two subdivisions (the other neighbors of $v$), we see that $G-v_1,v_2$ is planar.

\begin{figure}[ht]
\begin{center}
\includegraphics[scale=0.5]{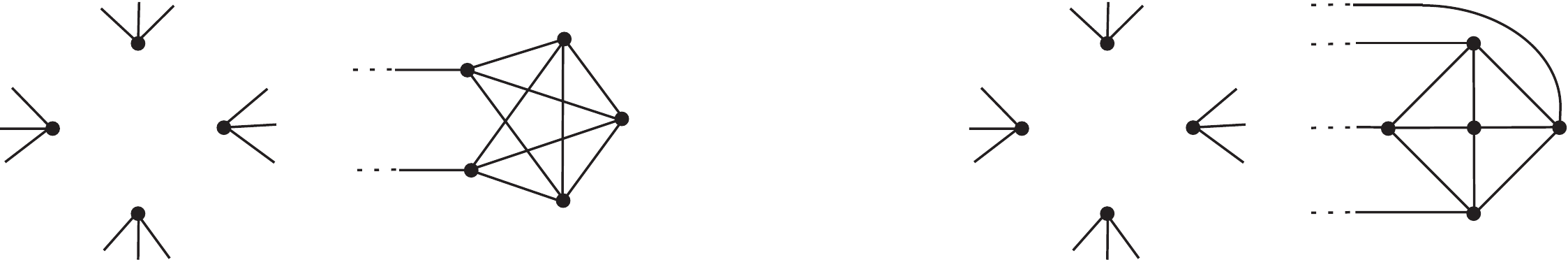}
\caption{There are two or four edges between $V_3$ and $V_4$.}\label{fig34part}
\end{center}
\end{figure}

We can now assume that the two degree $3$ vertices of $G$ have no common degree 4 neighbors.
Let $a$ be a degree 4 vertex that has a degree 3 neighbor. 
Then $G^* = (G-a)^s$ has degree sequence
$(3^4,4^5)$. Notice first that if $G^*$ has a degree 4 vertex $v$ that has three or more degree $3$ neighbors, then $(G^*-v)^s$ has at most $9$ edges and $5$ vertices and is planar. 
We claim that there is a degree 4 vertex in $G^*$, that has two neighbors of degree $3$. 
For suppose not and let $V_3$ denote the set of degree 3 vertices of $G^*$ and
$V_4$ those of degree 4. As the degree sums in the two parts are even, there are
an even number of edges between $V_3$ and $V_4$. If there were six or more, then, by pigeonhole, one of the degree 4 vertices would have two degree 3 neighbors, which is what we are trying to establish. If there were no edges in between, $G^* = K_4 \sqcup K_5$ is apex, a contradiction. So there are two or four edges between 
$V_3$ and $V_4$. (See Figure~\ref{fig34part}.)
In either case, removing a degree $4$ vertex that has a degree 3 neighbor will result in a planar graph.

\begin{figure}[ht]
\begin{center}
\includegraphics[scale=0.5]{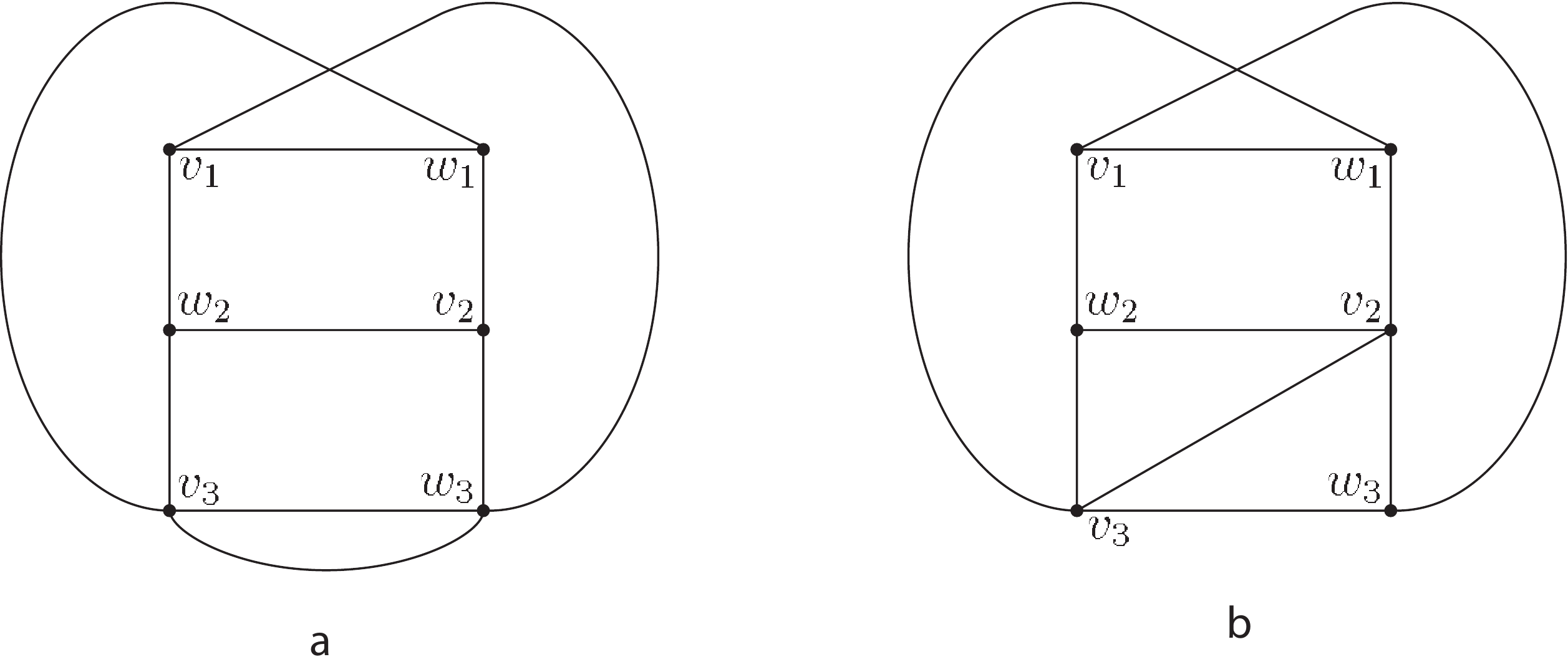}
\caption{Two non-planar $(6,10)$ graphs.}\label{figK33pe}
\end{center}
\end{figure}

So, let $b\in V(G^*)$ be a degree 4 vertex with two degree 3 neighbors. 
Moreover, $a$ and $b$ have a common neighbor, as otherwise $b$ has two degree $3$ neighbors in $G$.
Now, $G^* -b$ will be formed by subdividing two edges of a $(6,10)$ graph $G'$ having degree sequence $(3^4,4^2)$.
Since our assumption implies that $G'$ is non-planar, $G'$ is one of the two graphs obtained
by adding an edge to $K_{3,3}$ (see Figure~\ref{figK33pe}).

Assume that $G'$ is the multigraph $K_{3,3}+e$ shown in Figure~\ref{figK33pe}a. Since $G$ was a simple graph, there is at least one subdivision on one of the paired edges.
If $a$ and $b$ are not near the same edge in the paired edges, then removing from $G$ vertices $v_3$ and $w_3$ of $G'$ results in a planar graph, since the graph is essentially a subdivision of the $4$-cycle $v_1w_1v_2w_2$ along
with two more vertices that are not adjacent to one another.

Next, suppose $a$ and $b$ are adjacent to the same edge in the pair, but attach
to the edge at two different vertices formed by subdividing that edge twice.
By generalizing the argument of Lemma~\ref{lemapexpath}, we claim that
both $a$ and $b$ must have paths to each of the vertices in $G'$ independent of the other vertices of $G'$. For example, without loss of generality and referencing Figure~\ref{figK33pe}a, if no such
path from $a$ to $v_2$ exists, then $G-b,w_3$ must be planar. Indeed, place $a$ in the region of  $G'-w_3$ bounded by the cycle $v_1w_1v_3w_2$.
We can argue similarly for $b$.
Recall that $G-a,b$ is obtained from $G'$ by exactly three edge subdivisions. 
Also, when we add $b$ to $G'$, it is adjacent to two vertices formed by subdivision and two vertices of degree $3$. Using Lemma~\ref{lemda4}, we can assume $b$
is near $v_1w_1$ via a subdivision of that edge and also adjacent to $v_2$
and $w_2$ (by the symmetry of $G'$). Note that since $\Delta(G) = 4$, 
there is now no way to make paths from $a$ to $v_2$ and $w_2$ that avoid the other vertices of $G'$. 

We conclude that $a$ and $b$ attach at the same vertex of one of the paired edges of $G'$. Then as above, we can assume that $b$ is near the edge $v_1w_1$ and
adjacent to $v_2$ and $w_2$. Then those two vertices have degree $4$ and are
not adjacent to $a$. As there remains a single subdivision of $G'$, it must be on the
edge $v_2w_2$. So, $a$ is near that edge which
forces $a$ to be adjacent to $v_1$ and $w_1$. This graph is $N_{11}$.

Now assume that $G'$ is the simple graph illustrated in Figure~\ref{figK33pe}b. The graph $G'-v_3$, shows us that both $a$ and $b$ are near $w_1$, $w_2$, and $w_3$. Similarly, $G'-w_3$ shows us that they are near $v_3$
and $v_2$. Recall that $b$ is adjacent to two of the degree 3 vertices of $G'$ as well
as two vertices formed by subdividing edges of $G'$.

Suppose $b$ is adjacent to $v_1$ in $G-a$. Then $b$ is adjacent to one of the $w_i$ for $i\in \{1,2,3\}$, and by symmetry, we may assume $w_1$. Since $b$ is also near the other four vertices in $G'$, we may assume $b$'s other neighbors are 
vertices resulting from subdivisions of the edges $w_2v_2$ and $v_3w_3$. Since $a$ and $b$ share at least one neighbor, we may assume (without loss of generality) that $a$ is adjacent to the same vertex formed by subdividing
$w_3v_3$ of $G'$. 

There must be an additional subdivision of $G'$ giving a 
neighbor of $a$. Since $\Delta(G) = 4$, the remaining two neighbors of
$a$ are drawn from $\{w_2, w_3 \}$ and the vertex on $v_2w_2$ resulting from its subdivision. 
Suppose $a$ is adjacent to $w_2$ and $w_3$. As it must also be near $v_2$ and $w_1$, it is also adjacent to a vertex formed by a subdivision of the edge $v_2w_1$ in $G'$. However, in this case $v_2$ has
two neighbors of degree $3$, a possibility ruled out at the beginning of the proof.

So assume that $a$ shares two neighbors with $b$, the two vertices formed by subdividing $v_2w_2$ and $v_3w_3$, and is adjacent to exactly one of $w_2$ and $w_3$, say $w_3$. Now, $a$ must be near $w_1$ but if it is adjacent to
a vertex formed by the subdivision of $v_1w_1$ or $v_3w_1$, we again have the case of a degree $4$ vertex with two degree 3 neighbors  ($v_1$ and $v_3$ respectively). So it must be that $a$ is adjacent to a vertex resulting from subdivision of the edge
$w_1v_2$. In this case, let $x$ denote the common neighbor of $a$ and $b$ that is also a neighbor of $v_3$ and $w_3$. Then $G-x,w_3$ is planar. 
This shows that $b$ is not adjacent to $v_1$. A similar argument starting with adding $a$ instead of $b$ shows that $a$ is also not adjacent to $v_1$, at least in the
case where $a$ and $b$ share exactly one neighbor.

So we know that $b$ is not adjacent to $v_1$ in $G'$. Then without loss of generality it is adjacent to $w_2$ and $w_3$. So, $a$ is adjacent to $w_1$ or $v_1$. If $a$ is adjacent to $v_1$, then $a$ shares two neighbors with $b$. 
In other words, the vertices created by subdivisions in going
from $G'$ to $G-a,b$ that are neighbors of $b$ are also neighbors of $a$.
Since both $a$ and $b$ are near $w_1$, suppose they are adjacent to a vertex resulting from subdivision of the edge $v_1w_1$. Then since $a$ is near
$w_2,w_3,v_2,$ and $v_3$, we may assume $a$ is adjacent to vertices resulting from subdivisions of the edges $w_2v_2$ and $v_3w_3$ and that $b$ is adjacent to one of these. However, in either case $G$ has a degree $4$ vertex
with two degree $3$ neighbors ($v_3$ and $v_2$ respectively). 

So suppose instead that $a$ and $b$ are adjacent to a vertex produced by a subdivision of the edge $v_2w_1$ (The symmetric case using instead the edge $v_3w_1$ will be similar.) Since $a$ is 
near $v_3$, it must be adjacent to a vertex formed by subdivision of the edge $w_2v_3$ or $w_3v_3$ (the other two options will not allow $a$ to be near both $w_2$ and $w_3$). Without loss of generality it is $w_3v_3$. 
Moreover, this forces $b$ to share this
neighbor, as otherwise $v_3$ will have two degree $3$ neighbors in $G$. The final neighbor of $a$ makes it near $w_2$ but cannot lie on $v_1w_2$ or $v_3w_2$ 
lest we again have a vertex of degree $4$ with
two degree $3$ neighbors. So $a$ is adjacent to a vertex on the $w_2v_2$ edge. This is again $N_{11}$.

Finally, assume that neither $a$ nor $b$ is adjacent to $v_1$ in $G'$, $b$ is adjacent to $w_2$ and $w_3$, and $a$ is adjacent to $w_1$. The degree 3 vertices in $G$ are then $v_1$ and the one adjacent to $a$ formed by a subdivision of an edge in $G'$. Then the two subdivision vertices adjacent to $b$ must also be adjacent to 
$a$. Since $b$ is near $w_1$, assume first that $b$ is adjacent to a subdivision on the edge $v_1w_1$ in $G'$. Then the only way to make $b$ near both $v_2$ 
and $v_3$ is by making it adjacent to a vertex formed by subdividing that edge. As $a$ is also adjacent to that vertex, there is no way to make $a$ near both $w_2$ 
and $w_3$.
So without loss of generality $b$ (hence $a$) must be adjacent to a subdivision vertex on the edge $v_2w_1$ (as the symmetric case where $a$ and $b$ are adjacent to $v_3w_1$ is similar). Notice now that since $a$ is near both
$w_2$ and $w_3$ either $w_2$ or $w_3$ will share a degree $3$ neighbor with $a$. However, since they are both also neighbors of $v_1$, $G$ will have a degree $4$ vertex with two degree $3$ neighbors and cannot
be $2$--apex.
\end{proof}

\begin{prop} If $G$ is $(11,21)$ MMN2A, then $G$ is in the Heawood family.
\end{prop}

\begin{proof}
Assume that $G$ is an $(11,21)$ MMN2A graph. As we did in the previous cases, we may assume that the maximal vertex degree of $G$ is $6$ or less. Further, if $G$ has more than one vertex of degree $6$, then $G$ is not 
MMN2A, since it must be the case that one of the degree $6$ vertices has a degree $3$ neighbor and removing such a vertex leaves one with a graph that simplifies to a graph that has no more than $14$ edges, hence is not NA by Theorem~\ref{thmMMNA}.
This leaves us with the following degree sequences to consider:
$(3^7,5^3,6)$, $(3^6,4^2,5^2,6)$, $(3^5,4^4,5,6)$, $(3^4,4^6,6)$, $(3^6,4,5^4)$,$(3^5,4^3,5^3)$, $(3^4,4^5,5^2)$, $(3^3,4^7,5)$, and $(3^2,4^9)$.

We can throw out the first three sequences, since it is clear that the degree $6$ vertex must have a neighbor of degree $3$ and we find ourselves in the same situation as we were in at the beginning of this proof. 
Five of the remaining six sequences do in fact lead to an MMN2A graph and 
are treated in the five lemmas above. 

This leaves only the degree sequence $(3^6,4,5^4)$. Suppose $G$ is a MMN2A 
graph with this degree sequence. Each degree 5 vertex $v$ has
at most one degree 3 neighbor as otherwise $G-v$ simplifies to a graph 
of at most 14 edges and is not NA by Theorem~\ref{thmMMNA}.
This implies that the vertices of degree $4$ and $5$ when considered separately, induce a $K_5$ subgraph, with four of the vertices having other neighbors in $G$. Choose $a,b\in V(G)$ such that $\deg(a)=\deg(b)=5$,
and consider $G-a,b$. Observe that the induced $K_5$ subgraph becomes a $K_3$ subgraph when $a$ and $b$ are removed and only two of its three vertices have
neighbors in the rest of $G-a,b$. This means $(\Gab)^s$ has at most eight edges
and is planar, a contradiction. 
Therefore there is no $(11,21)$ MMN2A graph $G$ with
degree sequence $(3^6,4,5^4)$. Together with our five lemmas, this completes the proof.
\end{proof}

\section{10 vertex graphs}

In this section we prove that a $(10,21)$ MMN2A graph is in the Heawood family. This is a corollary of the following proposition, originally proved in~\cite{BM}.

\begin{prop}  \label{prop1}%
Let $G$ be a graph with either $|V(G)| \leq 8$ or else $|V(G)| \leq 10$ and $|E(G)| \leq 21$. If $G$ is  N2A and a $\YT$ move takes $G$ to $G'$, then $G'$ is also N2A.
\end{prop}

\begin{proof}
Since a graph of 20 or fewer edges is $2$--apex~\cite{Ma}, the only N2A graph with $|G| \leq 7$ is $K_7$, which has no 
degree three vertices. So, the proposition is vacuously true for graphs of order seven or less. 

Suppose $G$ is N2A with $|G| = 8$. As discussed in~\cite{Ma}, $G$ must be IK and we refer to the classification
of such graphs due independently to~\cite{CMOPRW} and \cite{BBFFHL}. There are 23 IK graphs on eight vertices, but only four have a vertex 
of degree three. In each case, a $\YT$ move on that vertex results in $K_7$, which is also N2A. 

Again, graphs of size 20 or smaller are $2$--apex. So, we can 
assume $\|G\| = 21$ and $|G| \geq 9$. If $G$ is of order nine and N2A, then, by \cite[Proposition~1.6]{Ma}, $G$ is a Heawood graph (possibly with the addition
of one or two isolated vertices). A $\YT$ move results in the Heawood graph $H_8$ or $K_7 \sqcup K_1$, both of which are N2A.

This leaves the case where $|G| = 10$. Assume $G$ is a $(10,21)$ N2A graph that admits a $\YT$ move to $G'$. For a contradiction, suppose $G'$ is $2$--apex with 
vertices $a$ and $b$ so that $G' - a,b$ is planar. Let $v_0$ be the degree three vertex in $G$ at the center of the $\YT$ move and $v_1,v_2,v_3$ the vertices of the resultant triangle in $G'$. Since $G$ is N2A, it must be
that $\{v_1, v_2, v_3 \}$ is disjoint from $\{a,b\}$. Fix a planar representation 
of $G'-a,b$. The triangle $v_1v_2v_3$ divides the plane into two regions. Let $H_1$ be the induced subgraph on the vertices interior to the triangle and $H_2$ that of the vertices exterior. Then $|H_1|+|H_2| = 4$. Since $G$ is N2A, there is an obstruction to converting the planar representation of $G'-a,b$ into 
a planar representation of $\Gab$. This means that both $H_1$ and $H_2$ contain vertices adjacent 
to each of the triangle vertices $\{v_1, v_2, v_3 \}$. In particular, $H_1$ and $H_2$ each have at least one vertex. 

Suppose $|H_1| = |H_2| = 2$. The graph $G-b,v_1$ is non-planar, but, its subgraph $G-a,b,v_1$ is essentially a subgraph of $G'-a,b$ (with the addition of a degree two vertex $v_0$ on the edge $v_2v_3$) and we will use the same planar representation for $G-a,b,v_1$ that we have for $G'-a,b$. 

Since $G-b,v_1$ is not
planar, there's an obstruction to placing $a$ in the same plane. If we imagine 
putting $a$ outside of a disk in the plane that covers
$G-a,b,v_1$, we see that their is some vertex $w$ in an $H_i$ that is {\bf hidden} from $a$. That is, although there's an edge $aw \in E(G)$, there is no $a$-$w$ path in the plane that avoids $G - b,v_1$. It follows that there's a cycle in $G - b,v_1$ with $w$ interior and $a$ exterior the cycle.

Without loss of generality, the hidden vertex $w$ is in $V(H_1) = \{c_1,d_1\}$, say $w = c_1$. This means we can assume  that $c_1 v_2 d_1 v_3 $ is a $4$--cycle in $G$,
which, in the planar embedding of $G'-a,b$, is arranged with $c_1$ interior to the cycle $v_2d_1v_3$. However, since $G'-a,b$ is planar, this means $c_1$ is also hidden from $v_1$ and $c_1v_1$ is not an edge of the graph. 

A similar argument using $G-b,v_2$ allows us to deduce a $4$--cycle $c_2 v_1d_2 v_3$ using the 
vertices $c_2$ and $d_2$ of $H_2$ while showing $c_2 v_2 \notin E(G)$. However,
it follows that $G - b,v_3$ is planar, a contradiction.

So, we can assume $|H_1| = 3$ while $H_2$ consists of the vertex $c_2$ with $\{v_1, v_2, v_3 \} \subset N(c_2)$. Suppose $H_1$ also has  a vertex, $c_1$, that is adjacent to all three triangle vertices. As $G-b,v_1$ is non-planar, there's a vertex of $H_1$, call it $d_1$, that is hidden from $a$ such that $c_1 v_2 d_1 v_3$ is 
a cycle in $G$ and $d_1v_1 \notin E(G)$. Similarly, $G-b,v_2$ shows that $c_1 v_1 e_1 v_3$ is in $G$ and $e_1v_2$ is not, $e_1$ being the third vertex of $H_1$. Now, $G-b,v_3$ will be planar unless $d_1e_1 \in E(G)$. However, in that case,
contracting $d_1e_1$ shows that $G'-a,b$ has a $K_{3,3}$ minor and is non-planar, a contradiction.

If $H_1$ has no vertex $c_1$ that, on its own, is adjacent to the three triangle vertices, then either $H_1$ is
connected, or else it is not but has an edge $c_1d_1$ such that $\{v_1, v_2, v_3 \} \subset N(c_1) \cup N(d_1)$. But, in this latter case, we can rearrange the planar representation of $G'-a,b$ such that the third vertex of $H_1$ is 
exterior to the triangle, returning to the earlier case where $|H_1| = |H_2|=2$. 
So we will assume $H_1$ is connected.

Suppose $H_1$ is not complete, having only two edges $c_1 d_1$ and $d_1 e_1$. Again $G-b,v_1$ shows
that at least two vertices of $H_1$ are in $N(v_2) \cap N(v_3)$ and there are two cases depending on
whether or not $\{c_1, e_1 \} \subset N(v_2) \cap N(v_3)$. If both $c_1$ and $e_1$ are in the intersection, then we can assume $c_1$ is hidden from $a$, meaning $ac_1 \in E(G)$, but $c_1 v_1 \notin E(G)$. Actually, since $c_1$ is interior to the
cycle $v_2 e_1 v_3$, it follows that $d_1$ is as well and $d_1v_1 \notin E(G)$ either. 
Then $e_1$ is the unique vertex of $H_1$ adjacent to $v_1$ and 
$G - b, v_2$ is planar, which is a contradiction.

If $c_1$ and $e_1$ are not both in $N(v_2) \cap N(v_3)$, we can assume that
$c_1$ and $d_1$ are the common vertices with at most one of those adjacent to $v_1$. If $c_1v_1 \notin E(G)$, then $G-b,v_2$ shows $d_1v_1 e_1 v_3$ 
is in $G$ and $e_1v_2$ is not. But then $G - b,v_3$ is planar, a contradiction.
So, we can assume it must be $d_1$ that's hidden, meaning 
$ad_1$ is an edge and $d_1v_1$ is not. In this case, $G- b,v_2$ must be planar, a contradiction.

Finally, if $H_1 = K_3$, then a similar sequence of arguments shows that, in $G'$, the 
induced subgraph on $V(H_1) \cup \{v_1,v_2,v_3\}$ is the octahedron graph and that
$a$ and $b$ are both adjacent to the three vertices of $H_1$.
By counting edges, we see that, in fact, $a$ and $b$ each have degree three and we have accounted for all edges in $G'$. Applying the $\TY$ move to recover $G$, we observe that $G$ is $2$--apex (for example, $G - c_1, d_1$ is planar for any pair of vertices $c_1,d_1 \in V(H_1)$), a contradiction.

We've shown that assuming $G'$ is $2$--apex leads to a contradiction. 
Thus, the proposition 
also holds in the case $|G| = 10$, which completes the proof.
\end{proof}

\begin{cor} If $G$ is a $(10,21)$ MMN2A graph, then $G$ is in the Heawood family.
\end{cor}

\begin{proof} 
Suppose $G$ is $(10,21)$ MMN2A. Recall that $\delta(G) \geq 3$ as otherwise a vertex deletion or
edge contraction on a small degree vertex gives a proper minor that is also N2A. 

In \cite{Ma}, we showed that a graph of order nine is MMN2A if and only if 
it is in the Heawood family.
So, if $G$ has a degree three vertex, then apply a $\YT$ 
move at that vertex to get a graph $G'$.
Then, by Proposition~\ref{prop1} and the classification of MMN2A graphs of order nine, $G'$ is Heawood, whence $G$ is too.
So, we can assume $\delta(G) \geq 4$ which means the degree sequence of $G$ is  either $\{4^8,5^2\}$ or $\{4^9,6 \}$.

\begin{figure}[ht]
\begin{center}
\includegraphics[scale=.6]{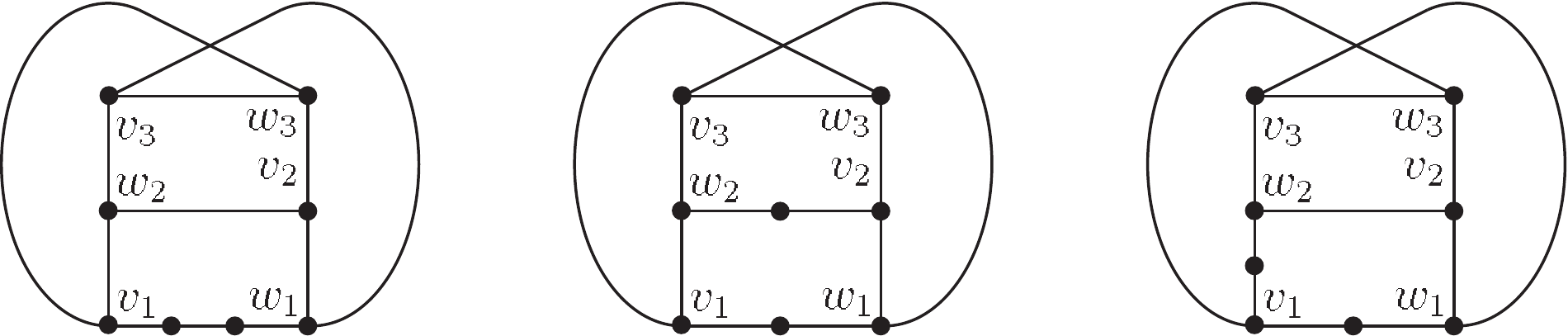}
\caption{The three non-planar (8,11) graphs of minimal degree at least two.}\label{figNP811}
\end{center}
\end{figure}

Suppose there are vertices $a$ and $b$ such that $\| \Gab \| =11$. 
Then at least one of $a$ and $b$ has degree five or six.
Since $\delta(G) = 4$, then $\delta(\Gab) \geq 2$ and $\Gab$ is one of 
one of the graphs of Figure~\ref{figNP811}.  In all three cases,
both $a$ and $b$ must be adjacent to both $v_3$ and $w_3$. For if, 
for example, $a$ and $v_3$ are not adjacent, then $G - b, w_3$ would be planar. 
But, if $a$ and $b$ are adjacent to both, then
$v_3$ and $w_3$ also have degree five in $G$, which contradicts the two given 
degree sequences for $G$. We conclude there is no choice
$a$ and $b$ such that $\| \Gab\| = 11$.

This means $G$ must have degree sequence $\{4^8, 5^2\}$ 
with the two vertices of degree five adjacent and
$\Gab$ a $(8,12)$ graph. There are two cases depending on whether or not $a$ and $b$ have a common neighbor
in $G$. Suppose first that $c$ is adjacent to both $a$ and $b$. In $\Gab$ vertex $c$ will have degree two and we can contract an edge on
$c$, to arrive either at a $(7,11)$ graph or else a multigraph with a doubled edge. Removing the extra edge if needed, 
let $H$ denote the resulting $(7,11)$ or $(7,10)$ graph.  

\begin{figure}[ht]
\begin{center}
\includegraphics[scale=.6]{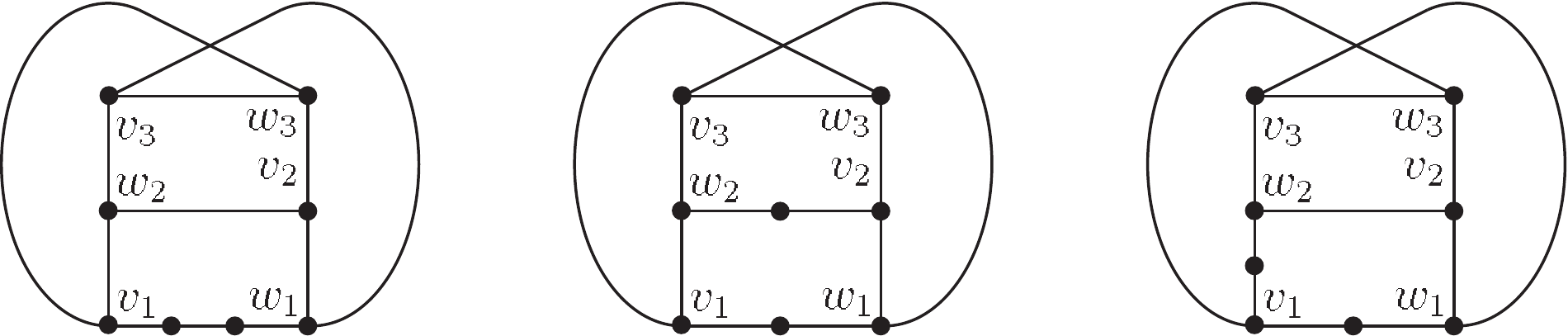}
\caption{The two non-planar (7,10) graphs of minimal degree at least one.}\label{figNP710}
\end{center}
\end{figure}

If $H$ is $(7,10)$, it is one of the two graphs of Figure~\ref{figNP710}. 
In the case of the graph on the left, the doubled edge must be that incident on the degree one vertex as $\delta(\Gab) \geq 2$.
But then the vertex labelled $v_1$ in the figure will have degree five in $\Gab$, contradicting our assumption 
that $a$ and $b$ were the only vertices of degree greater than four. So, we can assume $H$ is the graph to the right in the figure. 
Up to symmetry, the doubled edge of $H$ is either $uv_1$, $v_1w_2$, or $v_2w_2$. We'll examine the first case; the others are similar. 
Doubling $uv_1$ and adding back $c$ leaves $v_1$ of degree four in $\Gab$. Then $G - a,b,v_1$ simplifies to $K_{3,3} - v_1$. Since
$w_1$, $w_2$,  and $w_3$ all have degree three in $\Gab$, they each have exactly one of $a$ and $b$ as a neighbor in $G$. Suppose
$a$ is adjacent to $w_2$. Then $G - a, v_1$ is planar, contradicting $G$ being N2A. For the other two choices of edge doubling, once can again
delete a resulting degree four vertex along with $a$ or $b$ to achieve a planar graph. So $H$ being $(7,10)$ leads to a contradiction.

\begin{figure}[ht]
\begin{center}
\includegraphics[scale=.65]{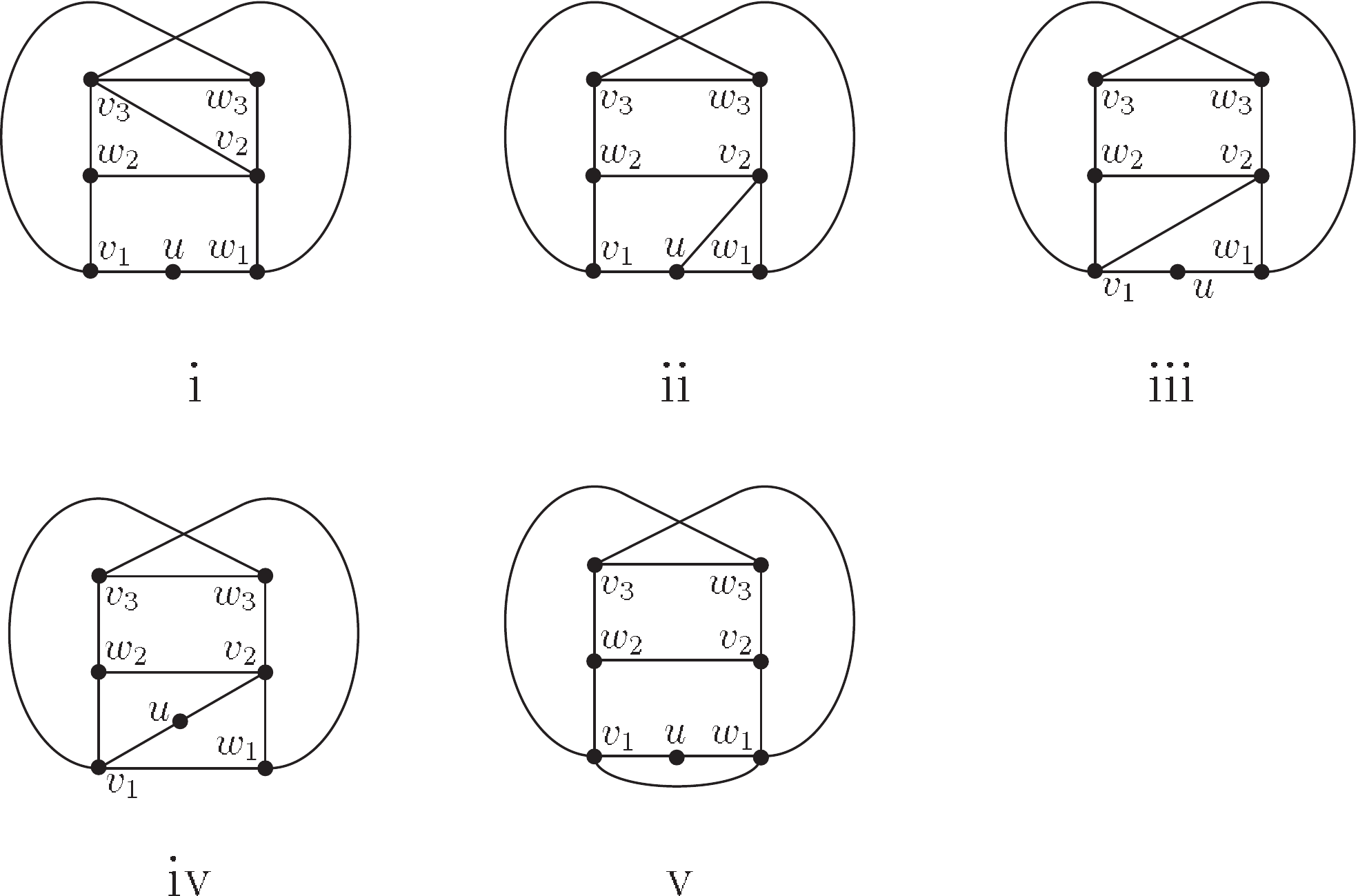}
\caption{The five non-planar (7,11) graphs of minimal degree at least two.}\label{figNP711}
\end{center}
\end{figure}

If $H$ is $(7,11)$, then $\delta(H) = \delta(\Gab) \geq 2$ and $H$ is one of the five graphs of Figure~\ref{figNP711}.
Here we use a similar approach. Deleting one of the degree four vertices of $H$, call it $x$, results in a graph 
$G - a,b,x$ that simplifies to $K_{3,3} - v_1$. Since each of the degree three vertices of $H$ is adjacent to exactly
one of $a$ and $b$, there will be an appropriate choice from those two, say $a$, such that $G - a,x$ is planar, which is a contradiction.
So, $H$ being $(7,11)$ is not possible and we conclude that there is no such vertex $c$ that is adjacent to both $a$ and $b$.

\begin{figure}[ht]
\begin{center}
\includegraphics[scale=.6]{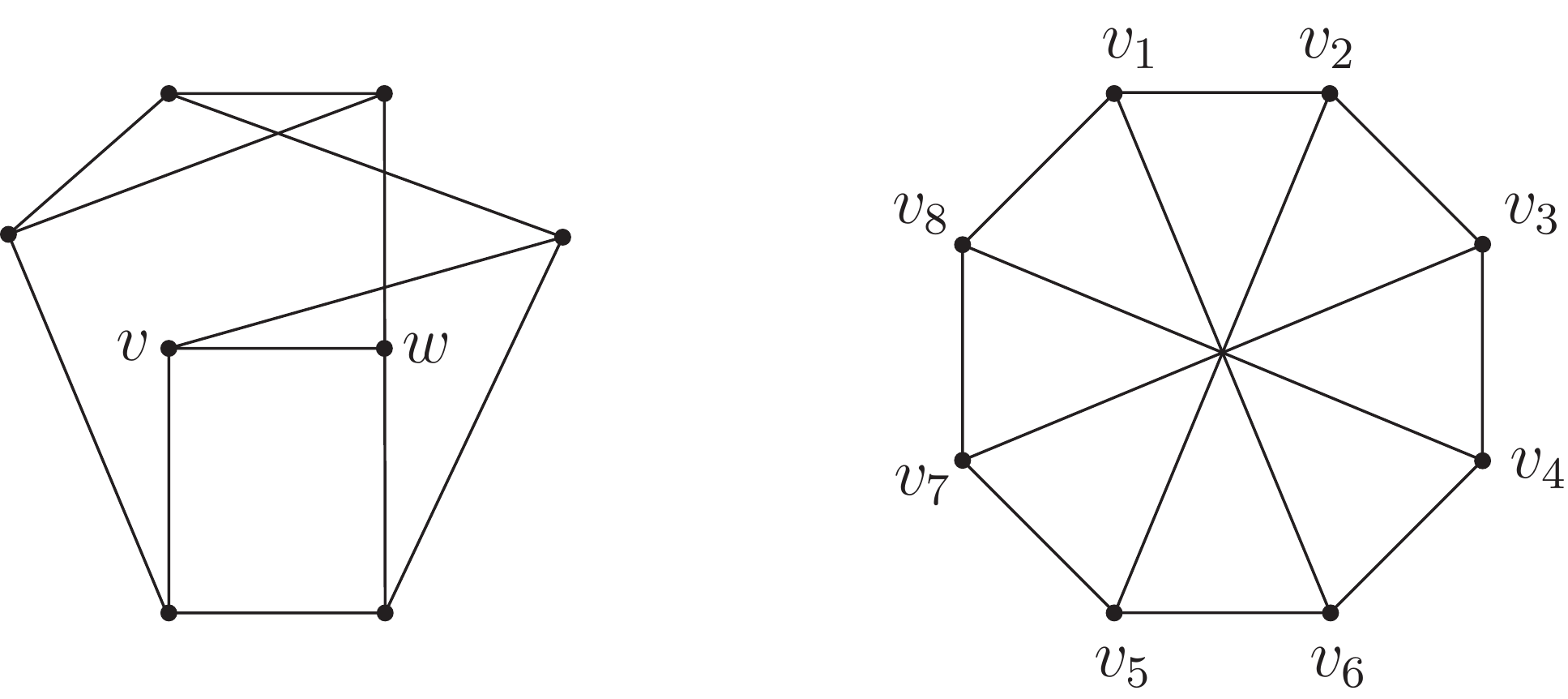}
\caption{The two non-planar cubic graphs of order eight}\label{fig8cubic}
\end{center}
\end{figure}

This means that $\Gab$ is a non-planar cubic graph (i.e., $3$-regular) on eight vertices. There are two such graphs, shown in Figure~\ref{fig8cubic}. If $\Gab$ is the graph to the left in Figure~\ref{fig8cubic}, note that the vertex labelled $v$ is adjacent to exactly one of $a$ and $b$, say $a$. Then $G - a,w$ is planar.

Finally, assume that $\Gab$ is the graph to the right in Figure~\ref{fig8cubic}.
Note that each vertex of $\Gab$ is adjacent to exactly one of $a$ and $b$ in $G$.
If $a$ and $b$ are adjacent to alternate vertices in the $8$--cycle (for example if $\{ v_1,v_3,v_5,v_7 \} \subset N(a)$ and $\{ v_2,v_4,v_6,v_8 \} \subset N(b)$), we obtain graph 20 of Figure~\ref{figHea}, a Heawood graph. If not, then we must have two consecutive vertices, say $v_1$ and $v_2$ that share the same neighbor in $\{ a,b \}$, say $a$. That is, we can assume $av_1, av_2 \in E(G)$.  Then $G - a, v_3$ is planar, contradicting $G$ being N2A.

In summary, if $G$ of order 10 is N2A with $\delta(G) > 3$, it must be graph 20 of the Heawood family. This completes the proof.
\end{proof}

\section*{Acknowledgments}

This research was supported in part by a Provost's Research and Creativity Award and a Faculty Development Award from CSU, Chico.

\end{document}